\documentclass[10pt,notitlepage,twoside,a4paper]{amsart}
 \usepackage{amsfonts}

\usepackage{amsmath,amssymb,enumerate}

\usepackage{epsfig,fancyhdr,color}

\usepackage{amssymb}
\usepackage{amsmath,amsthm}
\usepackage{latexsym}
\usepackage{amscd}
\usepackage{psfrag}
\usepackage{graphicx}
\usepackage[latin1]{inputenc}
\usepackage[all]{xy}
\usepackage[mathcal]{eucal}

\definecolor{NoteColor}{rgb}{1,0,0}

\renewcommand{\textsc}{\textcolor{red}}

\newtheorem{theorem}{\rm\bf Theorem}[section]
\newtheorem{proposition}[theorem]{\rm\bf Proposition}
\newtheorem{lemma}[theorem]{\rm\bf Lemma}
\newtheorem{corollary}[theorem]{\rm\bf Corollary}
\newtheorem*{theorem 1}{\rm\bf Proposition 1}
\newtheorem*{theorem 2}{\rm\bf Proposition 2}

\theoremstyle{definition}
\newtheorem{definition}[theorem]{\rm\bf Definition}

\theoremstyle{remark}

\newtheorem{example}[theorem]{\rm\bf Example}

\def\interieur#1{\mathord{\mathop{\kern 0pt #1}\limits^\circ}}

\title[Fenchel-Nielsen coordinates]{On Fenchel-Nielsen coordinates on Teichm\"uller spaces of surfaces of infinite type}

\author{D. Alessandrini}
\address{Daniele Alessandrini,  Institut de Recherche Math{\'e}matique Avanc\'ee,
Universit{\'e} de Strasbourg and CNRS,
7 rue Ren\'e Descartes,
 67084 Strasbourg Cedex, France}
\email{alessand@math.u-strasbg.fr}

\author{L. Liu}
\address{Lixin Liu, Department of Mathematics, Zhongshan University, 510275, Guangzhou, P. R. China}
\email{mcsllx@mail.sysu.edu.cn}

\author{A. Papadopoulos}
\address{Athanase Papadopoulos, Max-Planck-Institut f\"ur Mathematik, Vivatsgasse 7, 53111 Bonn, Germany and: Institut de Recherche Math{\'e}matique Avanc\'ee,
Universit{\'e} de Strasbourg and CNRS,
7 rue Ren\'e Descartes,
 67084 Strasbourg Cedex, France} \email{athanase.papadopoulos@math.unistra.fr}
\date{\today}

\author{W. Su}
\address{Weixu Su, Department of Mathematics, Zhongshan University, 510275, Guangzhou, P. R. China}
\email{su023411040@163.com}

\author{Z. Sun}
\address{Zongliang Sun, Department of Mathematics, Suzhou University, Suzhou 215006, P.R. China.}
\email{moonshoter@163.com}

\begin{document}

\begin{abstract} 
 We introduce Fenchel-Nielsen coordinates on Teichm\"uller spaces of surfaces of infinite type. The definition is relative to a given pair of pants decomposition of the surface. We start by establishing conditions under which any pair of pants decomposition on a  hyperbolic surface of infinite type can be turned into a  geometric decomposition, that is, a decomposition into hyperbolic pairs of pants.  This is expressed in terms of a condition we introduce and which we call {\it Nielsen convexity}.  This condition is related to Nielsen cores of Fuchsian groups. We use this to define the {\it Fenchel-Nielsen} Teichm\"uller space associated to a geometric pair of pants decomposition.  We study a metric on such a Teichm\"uller space, and we compare it to the quasiconformal Teichm\"uller space, equipped with the Teichm\"uller metric.   We study conditions under which there is an equality between these Teichm\"uller spaces
 and we study topological and metric properties of the identity map when this map exists.

\bigskip

\medskip

\medskip

\end{abstract}
\maketitle

\tableofcontents

\section{Introduction}\label{intro}
In this paper, $S$ is an oriented surface of infinite topological type. We assume that each boundary component of $S$ is a simple closed curve (that is, homeomorphic to a circle). 
We shall sometimes call a simple closed curve on $S$ a {\it circle}. The surface $S$ admits a topological pair of pants decomposition  $\mathcal{P}=\{P_i\}$ in which each $P_i$ is a generalized pair of pants; that is,  a topological sphere with three holes, where a {\it hole} is either a point removed or an open disk removed.  It will be useful to recall precisely the definition of a surface of infinite topological type, and we do this in Section \ref{sez:pants}. It follows from that definition  that any pair of pants decomposition of $S$ is necessarily countably infinite.

We shall consider hyperbolic structures on $S$ that satisfy a property that we call Nielsen-convexity, see Section \ref{sez:pants} for details. For every conformal structure on $S$, there is a unique hyperbolic metric in this conformal class that is Nielsen-convex. This metric is canonical in some sense, but in general it is not the Poincar\'e metric. A {\it generalized hyperbolic pair of pants} is a hyperbolic sphere with three geometric holes, where a geometric hole is either a geodesic boundary component or a puncture whose neighborhood is a cusp. We shall call a decomposition of a hyperbolic surface into generalized hyperbolic pair of pants glued along their boundary components a {\it geometric pair of pants decomposition}.
We show that the property of being Nielsen-convex is equivalent to the possibility of ``straightening'' every topological pair of pants decomposition of $S$ into a geometric one, that is, a decomposition by generalized hyperbolic pair of pants. 

More precisely, we prove the following:

\medskip

\noindent {\bf Theorem \ref{thm:pants dec}}.
{\it   Let $S$ be a hyperbolic surface with non-abelian fundamental group. We assume that $S$ contains all the boundary components of its metric completion that are circles. Then the following facts are equivalent:

\begin{enumerate}
\item $S$ can be constructed by gluing some generalized hyperbolic pairs of pants along their boundary components.
\item $S$ is Nielsen-convex.
\item Every topological pair of pants decomposition of $S$ by a system of curves $\{C_i\}$ is isotopic to a geometric pair of pants decomposition (i.e. for every curve $i$, if $\gamma_i$ is the simple closed geodesic on $S$ that is freely homotopic to $C_i$, then $\{\gamma_i\}$ defines a pair of pants decompositon).
\end{enumerate}
}

\medskip

(For a definition of a topological pair of pants decomposition in the setting of surfaces of infinite type, see the beginning of Section \ref{sez:pants}. )

\medskip
The analogue of this straigthening theorem is well known for surfaces of finite area, and it is commonly used to study the Teichm\"uller spaces of such surfaces, using the hyperbolic geometry point of view, as opposed to the complex-analytic. Theorem \ref{thm:pants dec} is now a tool for studying Teichm\"uller spaces of surfaces of infinite topological type.  In the sense we use here, Nielsen-convexity of a hyperbolic surface is the property that is needed as a generalizitaion of being a hyperbolic surface of finite area.

After proving Theorem \ref{thm:pants dec}, we shall study different ``Teichm\"uller spaces", that is, spaces of equivalence classes of marked hyperbolic structures (or conformal structures) on $S$, namely, the quasiconformal Teichm\"uller space $\mathcal{T}_{qc}$ and the Fenchel-Nielsen Teichm\"uller space $\mathcal{T}_{FN}$.

 We work in the setting of {\it reduced} Teichm\"uller spaces; that is, the equivalence relations on the sets of hyperbolic or on conformal structures are defined using homotopies that are not required to induce the identity on the boundary of the surface. For surfaces that do not have boundary components, the reduced and non-reduced Teichm\"uller spaces coincide. Since all the Teichm\"uller spaces that we use in this paper are reduced, we shall use, for simplicity, the terminology {\it Teichm\"uller space} instead of {\it reduced Teichm\"uller space}.

Unlike the corresponding spaces for surfaces of finite type, the various spaces that we consider do not depend only on the topological type of the surface $S$, but they also depend on other data, like the choice of a hyperbolic structure taken as a basepoint for Teichm\"uller space.  A Fenchel-Nielsen Teichm\"uller space also depends on the choice of the pair of pants decomposition $\mathcal{P}$ that we start with.

We shall especially consider Teichm\"uller spaces whose basepoints $H$ satisfy the following metrical boundedness property with respect to the fixed pair of pants decomposition pants decomposition by curves $\{C_i\}$: There is a constant $M$  such that $\ell_H(C_i) \leq M$ for all $i$.
 We call this property the {\it upper-bound condition}.

In the case of finite type surfaces, the various associated Teichm\"uller spaces coincide, although they give different equally interesting points of view on the same object. The Fenchel-Nielsen and the quasiconformal points of view, in the case of finite type surfaces, can be used to study different properties of Teichm\"uller spaces. For example, the description with Fenchel-Nielsen coordinates gives a practical way to see that the Teichm\"uller space is homeomorphic to a cell, and to compute its dimension.

Turning to the case of surfaces of infinite type, the Fenchel-Nielsen Teichm\"uller spaces that we define are infinite-dimensional, but as in the case of surfaces of finite type, they are also contractible. We equip these spaces with metrics that make them isometric to the sequence space $l^{\infty}$. 

 After describing the various Teichm\"uller spaces, we address the question of finding sufficient conditions under which these spaces coincide setwise, sufficient conditions under which the spaces are homeomorphic and the question of comparing their metrics, when thess spaces coincide.
 
One of the sesults that we obtain says that if the hyperbolic structure $H_0$, considered as the basepoint of a quasiconformal Teichm\"uller space $\mathcal{T}_{qc}(H_0)$ satisfies an upper-bound condition, then we have the set-theoretic equality $\mathcal{T}_{qc}(H_0)= \mathcal{T}_{FN}(H_0)$. From the metric point of view, we have the following:

\medskip

\noindent {\bf Theorem \ref{th:bilipschitz}.} {\it Let $H_0$ be a complete hyperbolic structure on $S$, and suppose that $H_0$ is upper-bounded with respect of some pair of pants decomposition $\mathcal{P}$. Then the identity map
$$j : \mathcal{T}_{qc}(H_0) \ni [f,H] \mapsto {(l_i(f,H), \theta_i(f,H))}_{i \in I} \in \mathcal{T}_{FN}(H_0)$$
between the spaces equipped with their respective metrics is a locally bi-Lipschitz homeomorphism. 
} 
\medskip

As $\mathcal{T}_{FN}(H_0)$ is isometric to the sequence space $l^\infty$, the above theorem gives a locally bi-Lipschitz homeomorphism between the quasiconformal Teichm\"uller space and the sequence space $l^{\infty}$.

This result should be compared with a recent result of A. Fletcher saying that the non-reduced quasiconformal Teichm\"uller space of any surface of infinite analytic type is locally bi-Lipschitz  to the sequence space $l^{\infty}$, see \cite{Fletcher} and the survey by Fletcher and Markovic \cite{FMH}.

   Note that the Fenchel-Nielsen Teichm\"uller space $\mathcal{T}_{FN}(H_{0})$, being isometric to the sequence space $l^{\infty}$, is complete.

The plan of the rest of this paper is the following.

Section \ref{sez:preliminaries-hyperbolic} contains some formulae from hyperbolic geometry that are useful in the sequel.
Section \ref{sez:preliminaries-quasiconformal} contains some preliminary material on quasiconformal mappings. 
Section \ref{sez:pants} contains the result about straigthening pair of pants decomposition into geometric pair of pants decompositions (Theorem 
\ref{thm:pants dec} stated above). In Section \ref{s:LS}, we introduce the quasiconformal Teichm\"uller space with its metric. 
In Section \ref{s:FN}, we discuss Fenchel-Nielsen coordinates and we introduce the Fenchel-Nielsen Teichm\"uller space with its metric.
In Section \ref{s:twist}, we start by recalling a few known facts about the Fenchel-Nielsen deformation. Then we give estimates on quasiconformal dilatations of twist maps that establish a lower bound for the quasiconformal distance between a hyperbolic structure and another one obtained by a Fenchel-Nielsen multi-twist (that is, a sequence of twists along a collection of disjoint and non-necessarily finite set of closed geodesics) in terms of the Fenchel-Nielsen distance of these two structures. These estimates use a boundedness condition on the set of lengths of curves. 
In Section \ref{s:upper},   under a boundedness condition on the lengths of closed geodesics of the pair of pants decomposition, we show that if the base structure $H_0$ satisfies an upper-bound condition, then we have the set-theoretic equality $\mathcal{T}_{qc}(H_0)= \mathcal{T}_{FN}(H_0)$, and we prove Theorem \ref{th:bilipschitz} stated above. In Section \ref{s:examples}, we collect some examples and couter-examples that show that some of the hypotheses that we made in this paper are necesssary.
    
\medskip
\noindent{\it Acknowledgement.} Liu, Su and Sun were partially supported by NSFC: 10871211, 10911130170.

\section{Preliminaries on hyperbolic geometry}\label{sez:preliminaries-hyperbolic}

 In this section, we collect some general results and formulae from hyperbolic geometry  that will be useful in the sequel.

A {\it geodesic} in a hyperbolic surface is an embedded arc whose image, in the coordinate charts, is locally a geodesic arc in the hyperbolic plane $\mathbb{H}^2$. 

The formulae in the next lemma concern hyperbolic right-angled hexagons; that is, hexagons in the hyperbolic plane whose angles are all right angles. 

 \begin{lemma} \label{hexagon} Let 
$a_1,b_3,a_2,b_1,a_3,b_2$  be in that order the lengths of the consecutive edges of a right-angled hexagon. (Thus, $a_i$ is opposite to $b_i$.) For $i=1,2,3$, let $h_i$ be the curve of shortest length joining $a_i$ to  $b_i$ (see Figure \ref{height}). 
Then, we have
\begin{equation}\label{formula:cosh1}
\cosh a_1 = -\cosh a_2\cosh a_3+\sinh a_2\sinh a_3\cosh b_1
\end{equation}
and
\begin{equation}\label{formula:cosh2}
\cosh^2 h_i =
  \displaystyle \frac{-1 + \cosh^2 a_1 + \cosh^2 a_2 + \cosh^2 a_3 + 2  \cosh  a_1 \cosh  a_2 \cosh  a_3}{\sinh^2 a_i}.
  \end{equation}

  \end{lemma}

\begin{proof}
For the proof of the first formula,  we refer to \cite[p. 85]{Fenchel}, and for the second formula, we refer to \cite[Formula 5.2]{BS}.

\end{proof}

         \bigskip
  \begin{figure}[!hbp]
\centering
\psfrag{1}{\small $a_1$}
\psfrag{2}{\small $b_3$}
\psfrag{3}{\small $a_2$}
\psfrag{4}{\small $b_1$}
\psfrag{5}{\small $a_3$}
\psfrag{6}{\small $b_2$}
\psfrag{7}{\small $h_1$}
\includegraphics[width=.35\linewidth]{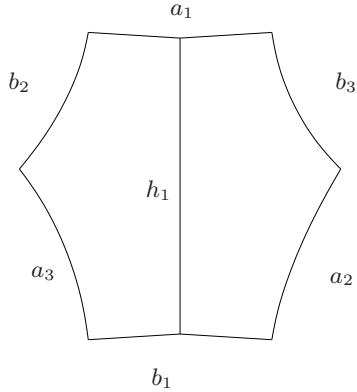}
\caption{\small {In this hyperbolic right-angled hexagon, $h_1$ is the shortest arc joining $a_1$ to $b_1$.}}
\label{height}
\end{figure}
\bigskip

The first formula in the above lemma allows to express $b_1$ in terms of $a_1,a_2,a_3$. It implies in particular that the isometry type of a right-angled hexagon 
is determined by the length of any three pairwise non-consecutive edges. We shall use this fact below.

Given a hyperbolic pair of pants $P$ having three geodesic boundary components, we shall call a {\it seam} of $P$ an arc joining two distinct boundary components and whose length is minimal among all lengths of arcs joining the given two boundary components. We recall that each pair of boundary components of $P$ are joined by a unique seam, that the endpoints of each seam intersect the boundary components with right angles, that the three seams of $P$ are disjoint and that they divide $P$ into the union of two congruent right-angled hexagons, see e.g. \cite{FLP}.

The next result is a version of a ``Collar Lemma".

\begin{lemma}   \label{lemma:collar}
Let $P$ be a hyperbolic pair of pants with three geodesic boundary components  $\partial_1,\partial_2,\partial_3$, of lengths respectively, $l_1,l_2,l_3$, and let $B(l)$ be the function 
$$B(l) = \frac{1}{2} \log\left(1+\frac{2}{e^l - 1}\right).$$
First, consider the three subsets $C_1,C_2,C_3$ of $P$ defined by
$$C_i = \{x \in P \ |\ d(x,\partial_i) \leq B(l_i)\}.$$
Then, the sets $C_i$ are annuli, and they are disjoint collar neighborhoods of the boundary components. 
\end{lemma}
\begin{proof}
Consider a right-angled hexagon with sides of length $a_1,b_3,a_2,b_1,a_3,b_2$ in that order, and for $i=1,2,3$, denote by $h_i$ the length of the shortest arc joining $a_i$ to  $b_i$. From Formula (\ref{formula:cosh2}) in Lemma \ref{hexagon}, we have:

  \begin{eqnarray*}
 \cosh b_1 &= &
\displaystyle  \frac{\cosh a_1 +\cosh a_2  \cosh  a_3}{\sinh a_2\sinh a_3} \\
&= & \displaystyle\frac{ \cosh a_1}{\sinh a_2\sinh a_3} + \coth a_2\coth a_3\\
& \geq & \displaystyle\coth  a_2\coth  a_3 \geq \coth a_2,
\end{eqnarray*}

and

  \begin{eqnarray*}
  \cosh^2 h_2 &=&
  \displaystyle \frac{-1 + \cosh^2 a_1 + \cosh^2 a_2 + \cosh^2 a_3 + 2  \cosh  a_1 \cosh  a_2 \cosh  a_3}{\sinh^2 a_2} \\
&=&
\displaystyle 1 + \frac{\cosh^2 a_1 + \cosh^2 a_3}{\sinh^2 a_2} + 2 \coth  a_2 \frac{ \cosh  a_1 \cosh  a_3}{\sinh a_2}\\
&\geq& \displaystyle 1 + 2  \frac{\coth  a_2}{\sinh a_2}.
\end{eqnarray*}

The above inequalities use the fact that if $x > 0$, we have $ \cosh  x \geq 1$, $\coth  x> 1 $ and $\sinh x > 0$.

 The pair of pants $P$ can be split into two congruent right-angled hexagons with non-consecutive side lengths \[a_i = \displaystyle \frac{l_i}{2}, \ i=1,2,3,\] the other three side lengths being $b_i, i=1,2,3$.

To prove the lemma, it suffices to prove that $\displaystyle \frac{b_1}{2}$, $\displaystyle \frac{b_3}{2}$ and $h_2$ are all $\geq B(l_2)$. The analogous properties for $B(l_1)$ and $B(l_3)$ will follow in the same way. 

From the above formula for $\cosh b_1$, we have
 \begin{eqnarray*}
 \frac{b_1}{2} &\geq & \frac{1}{2} \mbox{arcosh}\left(\coth \frac{l_2}{2} \right) \\
 & \geq &
 \frac{1}{2} \log\left(\frac{e^{l_2} + 1}{e^{l_2}-1}\right)\\
 &=& \frac{1}{2} \log\left(1 + \frac{2}{e^{l_2}-1}\right) \\
 &=& B(l_2).
 \end{eqnarray*}
 
The same argument applies to $\displaystyle \frac{b_3}{2}$.

For $h_2$, we have

  \begin{eqnarray*}
  h_2 &\geq& \mbox{arcosh}\left(\sqrt{1 + 2  \frac{\coth  a_2}{\sinh a_2}}\right)\\
  & =  &\displaystyle
\mbox{arcosh}\left(\sqrt{1 + 2 \left(\frac{e^{l_2} + 1}{e^{l_2}-1}\right) \left(\frac{2}{e^{a_2}-e^{-a_2}}\right)}\right) \\
&=&
 \displaystyle \mbox{arcosh}\left(\sqrt{1 + 4 \left(\frac{e^{l_2/2}(e^{l_2} + 1)}{(e^{l_2}-1)^2}\right)}\right) \\
 &\geq & \frac{1}{2} \log\left( 1 + 4 \left(\frac{e^{l_2/2}(e^{l_2} + 1)}{(e^{l_2}-1)^2}\right) \right)\\
 & \geq & B(l_2)
 \end{eqnarray*}
where we used the formula $\mbox{arcosh}(x)= \log\left(x+\sqrt{x^2-1}\right) \geq \log x$.
This completes the proof of Lemma \ref{lemma:collar}.
\end{proof}

\section{Conformal geometry and quasiconformal mappings}\label{sez:preliminaries-quasiconformal}

A Riemann surface is a one-dimensional complex sutructure on a topological surface.

This section contains a quick review of some results on quasiconformal mappings between Riemann surfaces, and on some conformal invariants. There are some good references on this subject, for instance \cite{Ah}, \cite{IT} and \cite{FM}. We start by recalling a few definitions.
 
Let $U$ be a domain in $\mathbb{C}$, and let $\mu:U\mapsto \mathbb{C}$ be an $L^\infty$ function with $\| \mu\|_\infty < 1$. Such a function is called a \emph{Beltrami differential}.

 A function $f:U \to \mathbb{C}$ is said to be $\mu$-\emph{conformal} if it is continuous, with first-order distributional derivatives in $L^2_{\mbox{\scriptsize loc}}$, and satisfies a.e. the \emph{Beltrami equation}:
$$\frac{\partial f}{\partial \bar{z}} = \mu(z) \frac{\partial f}{\partial z} $$
for some Beltrami differential $\mu$.

An important property of $\mu$-conformal functions is that a $0$-conformal function is holomorphic. 

The \emph{quasiconformal dilatation} (or, in short, the {\it dilatation}) of a $\mu$-conformal map $f$ is defined by
$$K=K(\mu)= K(f) = \frac{1+\|\mu(z)\|_\infty}{1-\|\mu(z)\|_\infty}\geq 1.$$
If $f$ is not $\mu$-conformal for any Beltrami differential $\mu$, we set $K(f) = \infty$.
A computation shows that the quasiconformal dilatation is multiplicative with respect to composition:
$$K(f\circ g) = K(f) K(g)$$
and that $K(f) = 1$ if and only if $f$ is holomorphic. A consequence is that $K(f)$ is invariant with respect to biholomorphic changes of coordinates both in the domain and in the range. 
In particular, if $f:S\to S'$ is a homeomorphism between two Riemann surfaces $S$ and $S'$, then for every coordinate patch $U \subset S$ such that $f(U)$ is contained in a coordinate patch of $S'$, the value $K(f|_U)$ is well defined  (we recall that its value is $\infty$ if it is not $\mu$-conformal for some $\mu$), and we define $K(f)$ as the supremum of all these values on set of coordinate patches covering $S$.  This quantity does not depend on the choice of the covering, and it is called the \emph{quasiconformal dilatation} of $f$. We shall say that $f$ is $K$-\emph{quasiconformal} (or simply \emph{quasiconformal}) if $K(f) \leq K < \infty$.

We now record two useful lemmas on conformal and quasiconformal homeomorphisms.

\begin{lemma}    \label{lemma:conformalcylinder}
Consider a hyperbolic surface isometric to a cylinder $S = \mathbb{H}^2 / <g>$, where $g$ is an isometry of hyperbolic type in the hyperbolic plane $\mathbb{H}^2$, and let $\gamma$ be the closed geodesic in $S$ which is the image of the invariant geodesic of $g$. Suppose that $\gamma$ has length $\ell$ and let $N$ be the neighborhood of width $b$ of $\gamma$; that is,
$$N = \{x \in S \ |\ d(x,\gamma) \leq b \}. $$
Then $N$ is conformally equivalent to the Euclidean cylinder $C \times I$ where $C$ is a circle of length $\ell$ and $I$ an interval of length $s = 4 \arctan\left(\frac{e^b-1}{e^b+1}\right)$.
\end{lemma}
\begin{proof}
We work in the upper- half plane model of the hyperbolic plane. The hyperbolic distance between two points is
$$d(z,w) = \log\frac{(z-w^*)(w-z^*)}{(w-w^*)(z-z^*)} $$
where $z^*,w^*$ are the endpoints on $\partial \mathbb{H}^2= \mathbb{R}^2\cup\{\infty\}$ of the geodesic joining $z$ and $w$, and where  $z^*,z,w,w^*$ are in that order on that geodesic.

For $z=i$ and $w=ri$ with $r > 1$, we have $z^* = 0, w^* = \infty$ and $d(i,ri) = \log r$. From this equality, we obtain $r = e^{d(i,ri)}$. 

If $z = i$ and $w = e^{i(\frac{\pi}{2}+\theta)}=-\sin\theta + i \cos\theta$ with $\theta > 0$, then $z^* = 1, w^* = -1$ and 
\[d(i,e^{i(\frac{\pi}{2}+\theta)}) = \log\frac{(i+1)(-1-\sin\theta+ i \cos\theta)}{(i-1)(1-\sin\theta+ i \cos\theta)} = \log\frac{1+\sin\theta}{\cos\theta} = \log\frac{1+\tan\frac{\theta}{2}}{1-\tan\frac{\theta}{2}},\] 
where the last equality comes from the tangent half-angle formula. By setting $d = d(i,e^{i(\frac{\pi}{2}+\theta)})$, we obtain 
\begin{equation}\label{formula:theta}
\theta = \theta(d) = 2 \arctan\left(\frac{e^d -1}{e^d + 1}\right).
\end{equation}

Now suppose that the invariant geodesic of $g$ is the Euclidean half-line $\{ x = 0, y > 0\}$. A fundamental domain for $S$ is 
$$\{z \ |\ \mbox{Im}(z) > 0, 1 \leq |z| \leq e^\ell\}.$$
A fundamental domain for $N$ is
\[\left\{\rho e^{i(\frac{\pi}{2}+\theta)} \ |\  1 \leq |z| \leq e^\ell, |\theta| \leq \theta(b)\right\}.\]
The  complex logarithm maps this domain to
\[\left\{ x + iy \ |\ 0 \leq x \leq \ell, \frac{\pi}{2} - \theta(b) \leq y \leq \frac{\pi}{2}+\theta(b)\right\},\]
which is a rectangle with side lengths $\ell$ and $4 \arctan\left(\frac{e^b-1}{e^b+1}\right)$, as required.
\end{proof}

\begin{lemma}   \label{lemma:affinequasiconformal}
Given aconstant $A \in \mathbb{R}$, consider the (real) affine map
\[f:\mathbb{C} \ni x+iy \mapsto x+Ay + iy \in \mathbb{C}.\]
Then the quasiconformal dilatation of $f$ is
\[K(f) = 1 + \frac{1}{2}A^2 + \frac{1}{2}|A|\sqrt{4+A^2} \leq 1 +|A|\sqrt{4+A^2}.\]
\end{lemma}
\begin{proof}
We have: $\displaystyle \frac{\partial f}{\partial x} = 1$,  $\displaystyle\frac{\partial f}{\partial y} = A+i$,  $\displaystyle\frac{\partial f}{\partial z} = 1-\frac{1}{2}iA$,  $\displaystyle\frac{\partial f}{\partial \bar{z}} = \frac{1}{2}iA$.

The Beltrami coefficient of $f$ is
\[\mu=\mu(f) =\displaystyle  \frac{\partial f}{\partial \bar{z}}\big/ \frac{\partial f}{\partial z} = \frac{\frac{1}{2}iA}{1-\frac{1}{2}iA} = \frac{iA(2+iA)}{4+A^2}.\]
The modulus of $\mu$ is
\[|\mu(f)| = \frac{|A|\sqrt{4+A^2}}{4+A^2} = \frac{|A|}{\sqrt{4+A^2}}.\]
Finally, the quasiconformal dilatation of $f$ is
\[K(f) = \frac{1+|\mu(f)|}{1-|\mu(f)|} = \frac{\sqrt{4+A^2} + |A|}{\sqrt{4+A^2} - |A|} = 1 + \frac{1}{2}A^2 + \frac{1}{2}|A|\sqrt{4+A^2}.\]
\end{proof}

Now we record a result due to Bishop \cite{Bishop2002} on quasiconformal mappings between hyperbolic pairs of pants. We shall use this result to  construct quasiconformal mapppings with controlled dilatation between more general hyperbolic surfaces.

Let $P$ (respectively $P'$) be a hyperbolic pair of pants with three geodesic boundary components, denoted by  $\partial_1, \partial_2, \partial_3$ (respectively $\partial'_1, \partial'_2, \partial'_3$). For all $i\not=j$, let $p_{i,j}$ (respectively  $p'_{i,j}$) be the intersection point  with $\partial_i$ of the seam joining $\partial_i$ and $\partial_j$ (respectively $\partial'_i$ and $\partial'_j$),  see Figure \ref{seams}.

        \bigskip
  \begin{figure}[!hbp]
\centering
\psfrag{1}{\small $\partial_1$}
\psfrag{2}{\small $\partial_2$}
\psfrag{3}{\small $\partial_3$}
\psfrag{4}{\small $p_{1,2}$}
\psfrag{5}{\small $p_{2,1}$}
\psfrag{6}{\small $p_{2,3}$}
\psfrag{7}{\small $p_{3,2}$}
\psfrag{8}{\small $p_{1,3}$}
\psfrag{9}{\small $p_{3,1}$}
\includegraphics[width=.50\linewidth]{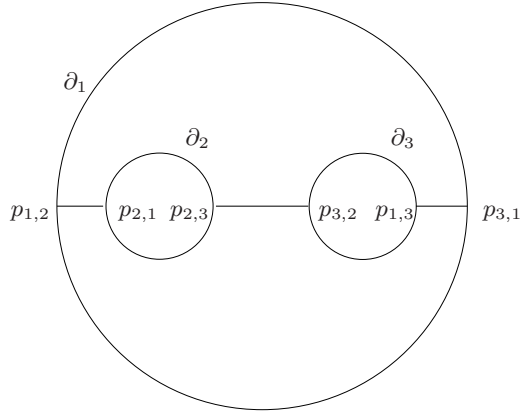}
\caption{\small {The three seams joining pairs of distinct boundary components in a hyperbolic pair of pants}}
\label{seams}
\end{figure}
\bigskip

\begin{theorem}[Bishop] \label{theorem:Bishop} For every real number $N>0$ there is a constant $C(N)$ such that if $P,P'$ are two hyperbolic pairs of pants with geodesic boundary components  $\partial_1, \partial_2, \partial_3$ and  $\partial'_1, \partial'_2, \partial'_3$ respectively,  with boundary lengths respectively $l_1,l_2,l_3$ and $m_1,m_2,m_3$, and satisfying $l_1,l_2,l_3,m_1,m_2,m_3 \leq N$, there exists a quasiconformal mapping $f:P \to  P'$ with the following properties:
\begin{enumerate}
\item for $i=1,2,3$, $f$ sends the boundary component $\partial_i$ to  the boundary component $\partial'_i$;
\item  $\displaystyle \log K(f) \leq 3 C(N) \max\left( \left|\log \frac{l_1}{m_1}\right|, \left|\log \frac{l_2}{m_2}\right|, \left|\log \frac{l_3}{m_3}\right| \right)$;
\item \label{p:3}
 for $i \neq j$, $f$ sends the point $p_{i,j}$ of $P$  to the point $p'_{i,j}$ of $Q$;
 \item for $i=1,2,3$,  the map $f$ is affine on each boundary component $\partial_i$; that is, it uniformly multiplies distances on this boundary component by the amount  $\displaystyle \frac{m_i}{l_i}$. 
\end{enumerate}
\end{theorem}
\begin{proof}
A particular case of this statement, namely, when $l_2 = m_2$, $l_3 = m_3$ and $|\log(l_1/m_1)| \leq 2$, is proved in \cite[Theorem 1.1]{Bishop2002}. By repeatedly applying that result and using the inequality $\log(1+x)\leq x$, we obtain the general case. Property (\ref{p:3}) follows from the fact that the map $f$  in \cite{Bishop2002} is constructed by gluing two maps between the right-angled hexagons obtained by splitting $P$ and $P'$ along the three seams $[p_{i,j},p_{j,i}]$ (respectively $[p'_{i,j},p'_{j,i}]$).

The constant $C(N)$ written here is the same as the one used in \cite{Bishop2002} (this is the reason for the factor $3$).   
\end{proof}

It is also useful to recall some formulae on conformal modulus and extremal length of quadrilaterals. Recall that a {\it quadrilateral} $H=D(z_1,z_2,z_3,z_4)$  in $\mathbb{C}$ 
consists of a Jordan domain
$H$ and a sequence of vertices $z_1,z_2,z_3,z_4$ on the boundary $\partial H$ following each other in that order so as to
determine a positive orientation of $\partial H$ with respect to $H$.
The vertices of the quadrilateral $D(z_1,z_2,z_3,z_4)$
divide its boundary into four Jordan arcs, called the sides of the quadrilateral. We shall call the arcs $\overline{z_1z_2}$ and
$\overline{z_3z_4}$ the $a$-sides, and the other two arcs the $b$-sides of $H$. Two quadrilaterals $D(z_1,z_2,z_3,z_4)$
and $D(w_1,w_2,w_3,w_4)$
are said to
be conformally equivalent if there is a conformal map from $D(z_1,z_2,z_3,z_4)$
 to $D(w_1,w_2,w_3,w_4)$
 which carries the point $z_i$ to the the point $w_i$ for $i=1,\ldots,4$.
It is a consequence of the Riemann mapping theorem that every quadrilateral $D(z_1,z_2,z_3,z_4)$ is conformally equivalent to a
quadrilateral $H(-1/k,-1,1,1/k)$ for some $0<k<1$ and where $H$ is the upper-half plane. From the Christoffel-Schwarz formula (see \cite{Ah}), we can see that each quadrilateral $D(z_1,z_2,z_3,z_4)$ is conformally equivalent to the Euclidean 
rectangle $R(0,a,a+ib,ib)$ with vertices $0,a,a+ib,ib$.  It is easy to see that two rectangles $R(0,a,a+ib,ib)$ and $R(0,a',a'+ib',ib')$
are conformally equivalent if and only if there is a similarity between them. Therefore, we can define the (conformal) modulus of the quadrilateral $D(z_1,z_2,z_3,z_4)$ by
$$\mathrm{mod}(D(z_1,z_2,z_3,z_4))=\frac{a}{b}.$$
It follows from the definition that the modulus of a quadrilateral is a conformal invariant and that
$\mathrm{mod}(D(z_1,z_2,z_3,z_4))=1/\mathrm{mod}(D(z_2,z_3,z_4,z_1))$.

The modulus of a quadrilateral $D(z_1,z_2,z_3,z_4)$ can also be characterized using extremal length, as follows. Let $\mathcal{F}$ be the family of curves in the domain $H$ joining the
$a$-sides of the quadrilateral. The {\it extremal length} of the family $\mathcal{F}$,
denoted by $\mathrm{Ext}(\mathcal{F})$, is defined by
$$\mathrm{Ext}(\mathcal{F}) = \sup_\rho \frac{{\inf_{\gamma\in \mathcal{F}}l_{\rho}(\gamma)}^2}{\mathrm{Area}(\rho)}$$
 where the supremum is taken over all conformal metrics $\rho$ on $H$ of finite
 positive area. It can be shown (see \cite{Ah}) that
 $$\mathrm{mod}(D(z_1,z_2,z_3,z_4))=\frac{1}{\mathrm{Ext}(\mathcal{F})}.$$

We close this section by recalling the following result due to Wolpert \cite{Wolpert-length}, which establishes a fundamental relation between hyperbolic geometry and quasiconformal homeomorphisms.

\begin{theorem}[Wolpert]\label{theorem:Wolpert}
 If $q:(S,H_1) \to (S,H_2)$ is a quasiconformal homeomorphism between two complete hyperbolic surfaces and if $\gamma$ is a loop in $H_1$, then 
$$\ell_{H_2}(q(\gamma)) \leq K(q) \ell_{H_1}(\gamma) $$
\end{theorem}

  \section{Straightening pair of pants decompositions}\label{sez:pants}
 
The decomposition of an infinite type hyperbolic surface into hyperbolic pairs of pants is a fundamental tool that we shall use in this paper, in order to study the Teichm\"uller spaces for surfaces of infinite type. In this section we give the relevant definitions and results about such a decomposition.  
 
To make things more precise, we recall that a \emph{surface} is a connected, second countable, Hausdorff topological space locally modelled on $\mathbb{R} \times \mathbb{R}_{\geq 0}$.  In particular, a surface is separable, it has partitions of unity, it is metrizable and it admits a triangulation. We refer to the paper \cite{Richards} for the classification of surfaces. In the present paper, we also assume that a surface is oriented.

A surface is said to be of \emph{(topological) finite type} if its fundamental group is finitely generated. Otherwise it is said to be of \emph{infinite type}. In the latter case, its fundamental group is a free group with a countable number of generators.

All the homotopies of the surface $S$ that we consider in this paper  preserve the punctures and preserve setwise the boundary components of $S$ at all times.

A \emph{pair of pants} is a surface whose interior is homeomorphic to a sphere with three distinct points deleted and whose boundary is a (possibly empty) disjoint union of circles. By gluing pairs of pants along their boundary components, we obtain more complex surfaces. This is called a pair of pants decomposition of the resulting surface. More precisely,
if $S$ is a surface whose boundary is a disjoint union of circles, a  \emph{pair of
pants decomposition} of $S$ is a system of pairwise disjoint simple closed curves 
$\mathcal{C} = {\{C_i\}}_{i \in I}$, such that
\begin{itemize}
\item $S \setminus \bigcup C_i$ is a disjoint union of pairs of pants without boundary;
\item it is possible to find a system of pairwise disjoint tubular neighborhoods of the 
curves $C_i$ in $S$ (where for the boundary curves of $S$, we mean, by a tubular neighborhood, a collar neighborhood).
\end{itemize}

As a consequence of the definition, the union $\bigcup C_i$ is a closed subset of $S$. 
All boundary components of $S$ are included in the curve system $\mathcal{C}$, and each such boundary component is in the frontier of exactly one of the pairs of pants. The curves of $\mathcal{C}$ 
that are not in the boundary of $S$ can be in the frontier of one or two pairs of pants. 
The set of indices $I$ is finite if the surface if of finite type, and it is countably 
infinite if the surface is of infinite type.

Any surface with non-abelian fundamental group and whose boundary is a (possibly empty) disjoint union of circles can be obtained by gluing a collection of pairs of pants along their boundary components. This is well known for surfaces of finite type. A proof of this statement in the general case can be found in \cite[Theorem 1.1, Thm. 2.1]{Alvarez}. (Note that the authors in the paper \cite{Alvarez} also need to use pieces they call {\it cylinders}, because in their definition, pairs of pants always have three boundary components, whereas, in our setting, we admit punctures.) The proof in \cite{Alvarez} uses hyperbolic geometry, but a purely topological proof of this fact can be written using the classification of surfaces given in \cite[Theorem 3]{Richards}. The number of pairs of pants needed for the decomposition is finite in the case of finite type surfaces and infinite in the case of infinite type surfaces. As, by definition, a surface is second countable, and as the pairs of pants in the decomposition have disjoint interiors, the number of pairs of pants needed is always at most countable.   
 
A \emph{hyperbolic metric} $h$ on a surface $S$ without boundary is a Riemannian metric of constant curvature $-1$. The pair $(S,h)$ is called a \emph{hyperbolic surface}. Given such a hyperbolic surface, denote by $\widetilde{S}$ the universal covering of $S$. It inherits a hyperbolic metric from $h$. The metric $h$ defines a developing map $D: \widetilde{S} \to \mathbb{H}^2$ and a holonomy representation $\rho:\pi_1(S) \mapsto \mbox{Isom}^+(\mathbb{H}^2)$ with the property that $D$ is a $\rho$-equivariant local isometry, see \cite{Thurston}.

A {\it geodesic} in a hyperbolic surfaces is an arc which, in the local coordinate charts, is the image of a geodesic arc of the hyperbolic plane.

We now discuss two different notions of completeness of hyperbolic structures.

Besides the usual notion of metric completenss, there is a notion of {\it geometric} completeness of a hyperbolic structure. This is in the general setting of geometric structures, and it means that the holonomy map of this structure is a homeomorphism onto the model space (which in the present case is the hyperbolic plane $\mathbb{H}^2$).
 We refer to Thurston's book \cite{Thurston} and notes \cite{Thurston-notes} for this notion of a geometrically complete geometric structure. It will be useful to distinguish carefully between these two notions of completenss, and we shall do so below whenever this is necessary.

A hyperbolic surface $S$ is said to be \emph{convex} if for every pair $x,y \in S$ and for every arc $\gamma$ with endpoints $x$ and $y$, there exists a geodesic arc in $S$ with endpoints $x$ and $y$ that is homotopic to $\gamma$ relatively to the endpoints, see \cite[Definition I.1.3.2]{CME}. A hyperbolic surface is convex if and only if its universal covering is convex, see \cite[Lemma I.1.4.1]{CME}. In that case the developing map $D: \widetilde{S} \mapsto \mathbb{H}^2$ is injective (see the beginning of the proof of \cite[Prop. I.1.4.2]{CME}), hence $\widetilde{S}$ is isometric to an open convex subset $C$ of $\mathbb{H}^2$. This implies that the holonomy representation $\rho:\pi_1(S) \mapsto \mbox{Isom}^+(\mathbb{H}^2)$ identifies $\pi_1(S)$ with a subgroup $\Gamma \subset \mbox{Isom}^+(\mathbb{H}^2)$ acting freely and properly discontinuously on the convex set $C$ satisfying $S = C / \Gamma$.

\begin{lemma}\label{lemma:subgroup}  The subgroup $\Gamma  \subset \mbox{Isom}^+(\mathbb{H}^2)$ is discrete with respect to the topology it inherits from the compact-open topology on $\mbox{Isom}^+(\mathbb{H}^2)$. Furthermore,  $\Gamma$ does not contain elliptic elements. In particular $\Gamma$ acts freely and properly discontinuously on $\mathbb{H}^2$.   
\end{lemma}
\begin{proof}
As the action of $\Gamma$ on $C$ is properly discontinuous, the orbit of a point $x \in C$ is discrete. Let $U$ be a neighborhood of $x$ not containing other points of the orbit. Consider the set
$$\{ \gamma \in \Gamma \ |\ \gamma(x) \in U\}.$$
This set is open for the compact-open topology, and it contains only the stabilizer of $x$ in $\Gamma$, that is reduced to the identity because the action is free on $C$. Hence $\Gamma$ is discrete.

Suppose that there exists an elliptic element $\gamma \in \Gamma$ fixing a point $x \in \mathbb{H}^2$. Take a point $y \in C$. Then the orbit $\{\gamma^n(y)\}$ is contained in $C$ because $C$ is invariant, and the convex hull of the orbit is also in $C$, since $C$ is convex. If $\gamma$ is not the identity, the point $x$ is in the convex hull of the orbit $\{\gamma^n(y)\}$, hence we have $x \in C$. But the action of $\Gamma$ on $C$ is free, a contradiction.   
\end{proof}

By Lemma \ref{lemma:subgroup}, $\Gamma$ acts freely and properly discontinuously on $\mathbb{H}^2$, hence we can consider the hyperbolic surface $\mathbb{H}^2 / \Gamma$.  Note that this surface is metrically complete (as a consequence of the fact that hyperbolic space is complete, and the group action is properly discontinuous). Also note that this surface is also geometrically complete. We have $S \subset \mathbb{H}^2 / \Gamma$, and this shows that every convex hyperbolic surface can be extended to a complete hyperbolic surface in a canonical way. We call this complete hyperbolic surface  the \emph{complete extension} of $S$.

Denote by $\overline{C}$ the closure of $C$ in $\mathbb{H}^2$, by $\partial C= \overline{C}\setminus C$ the frontier of $C$ in $\mathbb{H}^2$, and by $\partial C \cup \partial_\infty C$ the frontier of $C$ in $\mathbb{H}^2 \cup \partial_\infty \mathbb{H}^2$, with $\partial_\infty C \subset \partial_\infty \mathbb{H}^2$ and $\partial_\infty C \cap \partial   C = \emptyset$. 

Denote by $\Lambda \subset \partial_\infty \mathbb{H}^2$ the limit set of $\Gamma$, that is, the set of accumulation points of some (arbitrary) orbit (see \cite[Chapter 12]{Ra}), and by $\mbox{CH}(\Lambda) \subset \mathbb{H}^2$ the convex hull of $\Lambda$. As $C$ is invariant with respect to the action of $\Gamma$, then, by \cite[Theorem 12.1.2]{Ra}, $\Lambda \subset \partial_\infty C$, and $\mbox{CH}(\Lambda) \subset \overline{C}$. If $\pi_1(S)$ is not abelian, $\Gamma$ is not elementary; that is, it is not cyclic (by \cite[Theorem 5.5.9]{Ra}), hence $\mbox{CH}(\Lambda)$ has non-empty interior, $\mbox{int}(\mbox{CH}(\Lambda))$ (by \cite[Theorem 12.1.3]{Ra}). In this case we obtain a convex hyperbolic surface $\mbox{int}(\mbox{CH}(\Lambda)) / \Gamma$. We have $\mbox{int}(\mbox{CH}(\Lambda)) / \Gamma \subset S$, and this shows that every convex hyperbolic surface $S$ with non-abelian fundamental group (that is, a surface which is not homeomorphic to a disk or to a cylinder) contains a non-empty minimal (with respect to inclusion) convex hyperbolic surface. This surface is called the \emph{convex core} of $S$.

As $C$ is an open convex set in $\mathbb{H}^2$, $\overline{C}$ is homeomorphic to a surface with boundary, this boundary being  equal to the topological frontier $\partial C$ of $C$. Note that in general, the boundary of $C$  is not smooth. With respect to the metric induced from $\mathbb{H}^2$, $\overline{C}$ is the metric completion of $C$. As $\Gamma$ acts freely and properly discontinuously on $\overline{C}$, the quotient $\overline{S} = \overline{C} / \Gamma$ is a surface, the metric completion of $S$, whose boundary is the image in the quotient of the topological frontier of $C$. 

We shall say that a convex hyperbolic surface $S$ has \emph{geodesic boundary} if the boundary of the metric completion $\overline{S}$ is smooth and if every boundary component is a geodesic. Note that as $\overline{S}$ is metrically complete, boundary components are circles or complete geodesic lines. When lifted to the universal covering, the boundary components of $\overline{S}$ become complete geodesic lines.

\begin{lemma}
Let $S$ be a convex hyperbolic surface, with universal covering identified with a convex subset $C$ of $\mathbb{H}^2$. Then $S$ has geodesic boundary if and only if $C$ is the interior of the convex hull of $\partial_\infty C \subset \partial_\infty \mathbb{H}^2$. 
\end{lemma}
\begin{proof} Since the hyperbolic plane $\mathbb{H}^2$ is complete, the metric completion $\overline{C}$ of $C$ is also the closure of $C$ in $\mathbb{H}^2$. 
The ``if'' part is clear. For the ``only if'' part, note that $\overline{C}$ is the convex hull of $\partial C \cup \partial_\infty C$. But the frontier $\overline{C}- C$ of $C$ in $\mathbb{H}^2$ is a disjoint union of complete geodesic lines, hence every point of $\partial C$ is in the convex hull of $\partial_\infty C$.  
\end{proof}

Up to now, $S$ was a surface without boundary, while its metric completion $\overline{S}$ has a boundary that is a (possibly empty) disjoint union of circles and lines. \emph{From now on, when $S$ is a convex hyperbolic surface with geodesic boundary, it will be more comfortable to consider that all the boundary components of $\overline{S}$ that are circles are also in $S$}. 

We shall mainly be interested in convex hyperbolic surfaces with geodesic boundary satisfying an additional property. The definition is as follows:

\begin{definition}
A convex hyperbolic surface $S$ with geodesic boundary is \emph{Nielsen-convex} if every point of $S$ is contained in a geodesic arc with endpoints contained in simple closed geodesics in $S$.
\end{definition}
 Note that in this definition we use the convention we made above that the surface actually contains the closed geodesics that are on the boundary of its metric completion. 
 
We shall also use the following terminology: 

A {\it hyperbolic half-plane} in a hyperbolic surface is a subset isometric to a connected component of the complement of a geodesic in $\mathbb{H}^2$.

 A {\it funnel} is a subsurface isometric to the quotient of a hyperbolic half-plane
 by an isometry of hyperbolic type whose axis is the boundary of that half-plane.

A Nielsen-convex hyperbolic surface cannot contain hyperbolic half-planes. This is because a hyperbolic half-plane does not contain any closed geodesic, and if a geodesic enters a hyperbolic half-plane, it can never leave it. Thus, no point in a hyperbolic half-plane belongs to a geodesic arc with endpoints on a closed geodesic. Funnels always contain hyperbolic half-planes, hence Nielsen-convex hyperbolic surfaces contain no funnels. For a surface of finite type, being Nielsen-convex is equivalent to being convex with geodesic boundary and finite area (this follows from Theorem \ref{thm:pants dec} and Proposition \ref{prop:nielsen convex} below). For surfaces of infinite type it is not possible to require finite area. In the context of surfaces of infinite type, being Nielsen-convex is the property we use that replaces the property of being convex with geodesic boundary and finite area.

The building blocks of Nielsen-convex hyperbolic surfaces are the \emph{generalized hyperbolic pairs of pants}. These are pairs of pants equipped with a convex hyperbolic metric with geodesic boundary. Every topological hole in such a pair of pants corresponds to either a closed boundary geodesic or to a \emph{cusp}; that is, a surface isometric to the quotient of the region $\{z=x+iy \ \vert \ a<y\}$ of the upper-half space plane of the hyperbolic plane, for some $a>0$, by the isometry group generated by $z\mapsto z+1$.

We shall use the following:

\begin{lemma}  \label{lemma:pants are convex hull}
Let $P$ be a generalized hyperbolic pair of pants with at least one geodesic boundary component. Then $P$ is Nielsen-convex. 
\end{lemma}
\begin{proof}
We must show that each point in $P$ is on a geodesic arc whose endpoints are on the boundary of $P$.
Denote by $C$ the universal covering of $P$,  identified with a convex subset of $\mathbb{H}^2$. The set $C$ is closed in $\mathbb{H}^2$, and it is equal to $\mbox{CH}(\Lambda)$, where $\Lambda$ is the limit set of the holonomy group. If $x$ is in the interior of $C$, consider a geodesic $\gamma$ of $\mathbb{H}^2$ containing $x$ and intersecting $\partial C$. This implies that one of the endpoints of $\gamma$ in $\partial \mathbb{H}^2$ is outside $\Lambda$. There are two cases:

\noindent {\it Case 1.---} The other endpoint of $\gamma$ is also outside $\Lambda$. In this case, $\gamma$ cuts the boundary of $C$ on two different sides of $x$, and this shows that the image of $x$ in $P$ is on a geodesic arc whose endpoints are on the boundary of $P$.

\noindent {\it Case 2.---} The other endpoint of $\gamma$ is inside $\Lambda$. In this case, cince $\Lambda$ is nowhere dense in $\partial \mathbb{H}^2$ and since its complement is open in $\partial \mathbb{H}^2$ (see \cite[Theorem 12.1.9]{Ra}), we can choose a geodesic close to $\gamma$, containing $x$, and whose two endpoints are outside $\Lambda$.   In this case also, the image of $x$ in $P$ is on a geodesic arc whose endpoints are on the boundary of $P$.
\end{proof}

If $S$ is a convex hyperbolic surface with geodesic boundary, a  \emph{geometric pair of
pants decomposition} of $S$ is a pair of pants decomposition $\mathcal{C} = {\{C_i\}}_{i 
\in I}$, such that every curve $C_i$ in the decomposition is a simple closed geodesic, 
and every connected component of $S \setminus \bigcup C_i$ is isometric to the interior  of a generalized hyperbolic pair of pants.

If $S$ is a convex hyperbolic surface with geodesic boundary and $\mathcal{C} = \{C_i\}$ 
is a pair of pants decomposition of $S$, then it is always possible to find, for every 
$i$, a geodesic $\gamma_i$ that is freely homotopic to $C_i$ (see the proof of Theorem 
4.5 for details on how to do this). These geodesics do not form, in general, a 
geometric pair of pants decomposition. This is clear if the surface contains some 
funnels, a case that we can already see with finite type surfaces. With infinite type 
surfaces it is possible to construct even subtler examples of convex hyperbolic surfaces 
with geodesic boundary such that every topological puncture is a cusp (i.e. we are 
assuming that there are no funnels) that do not admit any hyperbolic pair of pants 
decomposition. To construct such an example, consider one of the tight flutes of 
the second kind constructed in \cite[Thm. 4]{Basmajian}, and take its complete 
extension. In that case, if we take a topological pair of pants decomposition $\{C_i\}$, 
and we construct the corresponding geodesics $\gamma_i$, then the union 
of the curves $\gamma_i$ is not closed in $S$, but these curves are adherent to a complete open 
geodesic bounding an hyperbolic half-plane in $S$. For every pair of pants of the 
decomposition $\{C_i\}$, there is a corresponding generalized hyperbolic pair of pants 
bounded by some of the curves $\gamma_i$, but these generalized hyperbolic pair of pants 
do not cover the whole surface, they leave out this hyperbolic half-plane. To guarantee 
that the straightening constructions work fine, it is necessary to add the hypothesis 
that the surface is Nielsen-convex.

The main aim of the rest of this section is to prove the following pair of pants decomposition theorem. The result is stronger than the result in \cite{Alvarez}, because here we show not only that pair of pants decomposition exist, but also that every topological decomposition into pairs of pants may be straightened into a geodesic decomposition.  

\begin{theorem}  \label{thm:pants dec}
Let $S$ be a hyperbolic surface with non-abelian fundamental group and that is not a pair of pants with three cusps. We assume as before that $S$ contains all the boundary components of its metric completion that are circles. Then the following properties are equivalent:
\begin{enumerate}
\item $S$ can be constructed by gluing some generalized hyperbolic pairs of pants along their boundary components.
\item $S$ is Nielsen-convex.
\item Every topological pair of pants decomposition of $S$ by a system of curves $\{C_i\}$ is isotopic to a geometric pair of pants decomposition (i.e. for every curve $i$, if $\gamma_i$ is the simple closed geodesic on $S$ that is freely homotopic to $C_i$, then $\{\gamma_i\}$ defines a pair of pants decompositon).
\end{enumerate}
\end{theorem}
\begin{proof}
(1) $\Rightarrow$ (2): Suppose that $S$ is constructed by gluing some generalized hyperbolic pairs of pants along their boundary components. 

We first show that $S$ is convex. If $x,y \in S$ and $\gamma$ is an arc with endpoints $x$ and $y$, then, by compactness of $\gamma$, and since each boundary component in the pair of pants decomposition is isolated from the other boundary components, the arc $\gamma$ intersects only a finite number of pairs of pants. The union of these pairs of pants is a hyperbolic surface of finite type which is metrically complete, with finite volume and with geodesic boundary. Hence it contains a geodesic arc with endpoints $x$ and $y$ which is homotopic to $\gamma$ relatively to the endpoints. Hence $S$ is convex. 

Now we show that $S$ has geodesic boundary. Consider the metric completion $\overline{S}$ of $S$. If $x_0 \in \partial \overline{S}$,
 we wish to show that $\partial \overline{S}$ is a geodesic in a neighborhood of $x_0$. If $x_0$ belongs to the boundary of a pair of pants, this is obvious. Otherwise, consider a sequence $x_i \in S$ converging to $x_0$. For every $i$, $x_i$ is contained in a pair of pants $P_i$, and $x_0$ is not (since we assumed that $x_0$ is not on the boundary of a pair of pants). Hence, if $\delta_i$ is a shortest geodesic arc with 
 endpoints $x_0$ and $x_i$, $\delta_i$ will leave the pair of pants $P_i$ at a closed geodesic $\gamma_i$, bounding $P_i$. The closed geodesics $\gamma_i$ are all disjoint from 
 $\partial \overline{S}$, and they get arbitrarily close to $x_0$. In the universal 
 covering, consider a lift $\widetilde{x_0}$ of $x_0$. This point lies on the boundary of the preimage of $\overline{S}$, a convex subset of $\mathbb{H}^2$. The preimeges of the $\gamma_i$ give a set of disjoint geodesics coming arbitrarily close to $\widetilde{x_0}$. This shows that $\partial \overline{S}$ is a geodesic in a neighborhood of $x_0$. Hence $S$ has geodesic boundary. 

Finally, since every generalized hyperbolic pair of pants with at least one geodesic boundary component is Nielsen-convex (Lemma \ref{lemma:pants are convex hull}) and since every point of $S$ is contained in a pair of pants, $S$ is Nielsen-convex.

(2) $\Rightarrow$ (3): Suppose that $S$ is Nielsen-convex, and consider a topological pair of pants decomposition of $S$, i.e. a set of simple closed curves $\{C_i\}$ such that $S \setminus \bigcup_i C_i$ is a disjoint union of topological pairs of pants. If $P$ is one of these pairs of pants, its frontier $\partial P$ in $S$ is the disjoint union of some of the curves $C_i$ (up to three). 

Consider the complete extension $\mathbb{H}^2 / \Gamma$ of $S$ and, for every $i$, construct the geodesics $\gamma_i$ in $\mathbb{H}^2 / \Gamma$ freely homotopic to $C_i$, using \cite[Theorem 9.6.6]{Ra}. Then note that by \cite[Theorem 12.1.7]{Ra}, all closed geodesics of $\mathbb{H}^2 / \Gamma$ are contained in the convex core of $S$, hence they are contained in $S$. 

If $P$ is one of the pairs of pants of the topological decomposition, then the geodesics freely homotopic to the components of $\partial P$ bound a generalized hyperbolic pair of pants, denoted by $\mathcal{P}$. We have to show that $S$ is the union of all these generalized hyperbolic pairs of pants. 

To see this, note that every closed geodesic $\gamma$ is contained in the union of these generalized hyperbolic pairs of pants. In fact $\gamma$ is contained in a finite union of topological pairs of pants $P_1 \cup \dots \cup P_n$ of the original decomposition. As $\gamma$ does not intersect the boundary of this finite union, it also does not intersect the boundary of their geodesic realizations $\mathcal{P}_1 \cup \dots \cup \mathcal{P}_n$, by \cite[Theorem 12.1.7]{Ra}, hence it is contained there.

Let $x \in S$. Then $x$ is contained in a geodesic arc $\alpha$ with endpoints $a,b$ contained in two closed geodesics $\gamma, \delta$. We have seen that $\gamma$ and $\delta$ are contained in the union of the pairs of pants $\mathcal{P}$. By the first part of the proof of the theorem, this union is convex, hence it contains a geodesic arc homotopic to $\alpha$ with endpoints $a,b$. This arc must be $\alpha$, hence $x$ is in the union of the pairs of pants.

(3) $\Rightarrow$ (1): This is obvious, as we already noted that every surface with non-abelian fundamental group has a topological pair of pants decomposition.
\end{proof}

\begin{proposition}  \label{prop:nielsen convex}
Let $S$ be a convex hyperbolic surface with geodesic boundary and that is not a pair of pants with three cusps. Then $S$ is Nielsen-convex if and only if it is the union of its convex core and of the closed geodesics on the boundary of the convex core.
\end{proposition}
\begin{proof}
Let $K$ denote the union of the convex core of $S$ and the closed geodesics on the boundary of this convex core.

If $S$ is Nielsen-convex, we need to show that it is contained in $K$. Let $x \in S$. Then $x$ is contained in a geodesic arc $\alpha$ with endpoints $a,b$ contained in two closed geodesics $\gamma, \delta$. Every closed geodesic is contained in $K$, and, since $K$ is a convex hyperbolic surface, it contains a geodesic arc homotopic to $\alpha$ with endpoints in $a,b$. This arc must be $\alpha$, hence $x$ is in $K$.

For the reverse implication, we have to show that $K$ is Nielsen-convex. Consider a topological pair of pants decomposition of $S$. As in the proof of the Theorem \ref{thm:pants dec}, consider the complete extension $\mathbb{H}^2 / \Gamma$ of $K$ and for every curve of the decomposition, construct the geodesic in $\mathbb{H}^2 / \Gamma$ freely homotopic to that curve. As above, all closed geodesics are contained in the convex core. This gives a geometric pair of pants decomposition of a subset of the convex core. Theorem \ref{thm:pants dec} implies that a surface that is the union of generalized hyperbolic pairs of pants is Nielsen-convex, hence it must contain the convex core. This shows that the convex core is Nielsen-convex.
\end{proof}

Let $S$ be a surface without boundary and with non-abelian fundamental group, equipped with a complex structure. By the uniformization theorem, the universal covering of $S$ is conformally equivalent to $\mathbb{H}^2$. The hyperbolic metric of $\mathbb{H}^2$ gives a canonical hyperbolic metric on $S$ that is conformally equivalent to the complex structure and that is complete in the metric sense. This metric is the celebrated \emph{Poincar\'e metric}. 

There is another naturally defined hyperbolic metric conformally equivalent to $S$, which is called the \emph{intrinsic metric}. If the complex structure is of the \emph{first kind}, (i.e. if the limit set for the action of $\pi_1(S)$ on $\mathbb{H}^2$ is the whole $\partial \mathbb{H}^2$, see \cite[\S 12.1]{Ra}), the intrinsic metric is the Poincar\'e metric. If the complex structure is of the \emph{second kind} (i.e. if the limit set is not the whole $\partial \mathbb{H}^2$), consider the embedding of $\mathbb{H}^2$ in the Riemann sphere $\mathbb{C}\mathbb{P}^1$ as the upper-half plane, such that $\partial \mathbb{H}^2 = \mathbb{R}\mathbb{P}^1$. The action of $\pi_1(S)$ on $\mathbb{H}^2$ extends in a natural way to an action on $\mathbb{C}\mathbb{P}^1$, and it is properly discontinuous and free on $\mathbb{C}\mathbb{P}^1 \setminus \Lambda$, where $\Lambda$ is the limit set. The quotient $D(S) =( \mathbb{C}\mathbb{P}^1 \setminus \Lambda )/ \pi_1(S)$ is a Riemann surface called the \emph{Schottky double} of $S$. The image of $\mathbb{R}\mathbb{P}^1 \setminus \Lambda$ in $D(S)$ is a curve system dividing $D(S)$ into two parts, one biholomorphically equivalent to $S$, and the other anti-biholomorphically equivalent to $S$. Hence we can consider $S$ as embedded in $D(S)$ in a natural way. Consider the Poincar\'e metric of $D(S)$. The restriction of this metric to $S$ is, by definition, the {\it intrinsic metric}. In any case, the intrinsic metric associated with a Riemann surface structure on $S$ is Nielsen-convex. In the following, we shall always consider we a Riemann surface as endowed with its intrinsic hyperbolic metric. This will allow us to consider decompositions of a Riemann surface into generalized hyperbolic pairs of pants.

Note that the metric completion $\overline{S}$ of the surface $S$ is a surface which, as a metric space, is complete, but it may have some boundary components that are not circles. (There are such examples in \cite{Basmajian}.) Recall that if $S$ is convex with geodesic boundary, we consider all boundary components of $\overline{S}$ that are circles as being in $S$. With this convention, the surface $S$ is complete if and only if $S=\overline{S}$; that is, if all the connected components of $\partial \overline{S}$ are circles.

We close this section by noting an important difference between the case of surfaces of finite type and that of surfaces of infinite type, concerning the question of metric completeness of a hyperbolic metric on such a surface equipped with a geometric pairs of pants decompositions. (Recall that by assumption, each hole of a hyperbolic pair of pants corresponds either to a closed geodesic or to a cusp.)  In the case of a surface of finite type, such a surfaces is automatically complete.  In the case of surfaces of infinite type, this is not the case, and examples of such hyperbolic surfaces are given in the paper \cite{Basmajian}. In fact, Basmajian gives in that paper examples of non-complete infinite-type hyperbolic surfaces which have no punctures and no boundary components at all.
 
 There is however a condition under which a hyperbolic metric on a surface of infinite type is complete, and we state it in the following lemma, since this condition will turn out to be important in the sequel.
\begin{lemma}  \label{lemma:complete}
Let $H$ be a hyperbolic structure on $S$ for which the length of each closed geodesic representing an element of $\mathcal{C}$ is bounded above by some constant that does not depend on the chosen element of $\mathcal{C}$. Then, $H$ is complete. \end{lemma}    

\begin{proof}
If the length of the boundary curves of a hyperbolic pair of pants are bounded above by some constant $M$, then, by Lemma \ref{lemma:collar},  there exists a positive constant $B(M)$ such that the distance between two boundary curves is at least $B(M)$. In particular, any closed ball of radius $B(M)$ in $S$ is contained in at most two pair of pants, hence it is compact. The lemma now follows from the Hopf-Rinow Theorem.
\end{proof}

 \section{The quasiconformal Teichm\"uller space}\label{s:LS}

We shall often use the formalism of  {\it marked} hyperbolic structures. In this setting, we are given a topological surface $S$; then, a {\it marked hyperbolic structure} on $S$ is a pair $(f,H)$ where $H$ is a surface homeomorphic to $S$ equipped with a hyperbolic structure, and  $f:S\to X$ is a homeomorphism. A marked hyperbolic structure on $S$ induces a hyperbolic metric on the surface $S$ itself by pull-back. Conversely, a hyperbolic structure on $S$ can be considered as a marked hyperbolic surface, by taking the marking to be the identity homeomorphism of $S$.

We recall the definition of the {\it quasiconformal Teichm\"uller space} equipped with the {\it quasiconformal metric} (or {\it Teichm\"uller metric}). For this definition, the surface is equipped with a Riemann surface structure (that is, a one-dimensional complex structure). Note that the surfaces we consider will be generally equipped with hyperbolic structures, and the complex structures are the ones that underly them.

   \begin{definition}
        Consider a Riemann surface structure $H_0$ on $S$. Its {\it  quasiconformal Teichm\"uller space}, $\mathcal{T}_{qc}(H_0)$, is the set of equivalence classes $[f,H]$ of pairs $(f,H)$, where $H$ is a Riemann surface homeomorphic to $S$, where the marking  $f:(S,H_0)\to (S,H)$ is a quasiconformal homeomorphism, and where two pairs $(f,H)$ and $(f',H')$ are considered to be equivalent (more precisely, {\it conformally equivalent})  if there exists a conformal homeomorphism $f'':(S,H)\to (S,H')$ that is homotopic to $f'\circ  f^{-1}$.
        The equivalence class of the marked Riemann surface $(\mathrm{Id},H_0)$ is the {\it basepoint} of $\mathcal{T}_{qc}(S_0)$. 
        \end{definition}

   To simplify notation, in the sequel, we shall denote by the same letter an element of a Teichm\"uller space that is, a marked hyperbolic (respectively conformal) structure and an isotopy class of a hyperbolic (respectively conformal) structure on $S$ (which is obtained by pull-back using the marking), and also a hyperbolic (respectively conformal) structure on $S$ in the given isotopy class.

         The space $\mathcal{T}_{qc}(H_0)$ is equipped with the {\it Teichm\"uller metric}, of which we also recall the definition.  Given two elements $[f,S]$ and $[f',S']$ of $\mathcal{T}_{qc}(S_0)$ represented by marked conformal surfaces $(f,S)$ and $(f',S')$,  their {\it quasiconformal distance} (also called  {\it Teichm\"uller distance}) is defined as
        \begin{equation}\label{eq:qc} d_{qc}([f,S],[f',S'])=\frac{1}{2}\log \inf \{K(f'')\}
        \end{equation}
        where the infimum is taken over all quasiconformal homeomorphisms $f'':S\to S'$ homotopic to $f'\circ  f^{-1}$.

We shall say that two marked Riemann surfaces $(f,H)$ and $(f',H')$ are {\it quasiconformally equivalent} if there exists a quasiconformal homeomorphism in the homotopy class $f'\circ f^{-1}$. It follows fom the definition that all the elements of the same Teichm\"uller space are quasiconformally equivalent.

        A Riemann surface $R$ is said to be of {\it finite conformal type} if it is obtained (as a Riemann surface) from a closed Riemann surface by removing  a finite number of points. (Thus, $R$ may have punctures, and a neighborhood of each such puncture is conformally a punctured disk). From the definition, a surface of finite conformal type is also of finite topological type.
        
        It is important to note that the definition of the Teichm\"uller distance given in (\ref{eq:qc}) is the same as that given in the case of surfaces of conformal finite type, but that for surfaces of infinite type, choosing different basepoints may lead to different spaces.
 Indeed, any homeomorphism between two Riemann surfaces of finite conformal type is homotopic to a quasiconformal one, but there exist homeomorphisms between surfaces of infinite conformal type that are not homotopic to quasiconformal homeomorphisms. Thus, in general, we have $\mathcal{T}_{qc}(H_0)\not=\mathcal{T}_{qc}(H_1)$, if $H_0$ and $H_1$ are two distinct hyperbolic structures on $S$.

\section{Fenchel-Nielsen coordinates and Fenchel-Nielsen Teichm\"uller spaces}\label{s:FN}

We shall consider Fenchel-Nielsen coordinates for spaces of equivalence classes of hyperbolic structures on $S$, associated to the fixed pair of pants decomposition $\mathcal{P}$, with its associated set of homotopy classes of curves $\mathcal{C}$. These parameters are defined in the same way as the Fenchel-Nielsen parameters associated to a geometric pair of pants decomposition for surfaces of finite type, and we shall use Theorem \ref{thm:pants dec} to insure that such a pair of pants decomposition exists . 

Let $x$ be a hyperbolic structure on $S$ which is Nielsen convex.

To each homotopy class of curves $C_i\in \mathcal{C}$, we associate a {\it length parameter} and a {\it twist parameter} in the following way: 

For each $i=1,2,\ldots$, the $x$-length parameter of $C_i$, $l_x(C_i)\in ]0,\infty[$, is the length of the $x$-geodesic in the homotopy class $C_i$.  (We note that this $x$-geodesic exists and is unique, in the finite as well as in the infinite-type surface case.)

The twist parameter is only defined if  $C_i$ is not the homotopy class of a boundary component of $S$.
To define the $x$-twist-parameter $\theta_x(C_i)\in\mathbb{R}$ of $C_i$, we first replace each curve in the collection $\mathcal{C}$ by its geodesic representative, an we obtain a geometric pair of pants decomposition (Theorem \ref{thm:pants dec}).  We can then consider the surface equipped with the hyperbolic structure $x$ as being obtained by gluing along their boundaries the collection of generalized hyperbolic pairs of pants in $\mathcal{P}$. If $C_i$ is not a boundary curve of $S$, then the twist parameter measures the relative twist amount along the geodesic in the class $C_i$ between the two generalized pairs of pants that have this geodesic in common (the two pairs of pants can be the same). 
 In fact, to define the twist parameter, one only needs to define it for the case of a surface of finite type, which is the union of one or two pairs of pants that are adjacent to the closed geodesic $C_i$. This surface is obtained either by gluing two generalized hyperbolic pairs of pants along two boundary components of the same length, or by gluing two boundary components of equal length of a single pair of pants. The twist amount per unit time along  the (geodesic in the class) $C_i$ is chosen to be proportional (and not necessarily equal) to arclength along that curve, in such a way that a complete positive Dehn twist along the curve $C_i$ changes the twist parameter by addition of $2\pi$. Thus, in some sense, the parameter $\theta_x(C_i)$ that we are using is an ``angle" parameter. 

A precise definition of the twist parametersis contained in \cite[Theorem 4.6.23]{Thurston}. In this description, for every curve $C_i$, we fix a small tubular neighborhood $N_i$, we fix an orientation of $C_i$, and we fix two points $x_i,y_i \in C_i$. For every pairs of pants, fix three disjoint arcs, each joining two different boundary components, with endpoints on the chosen points. Now consider a hyperbolic structure on $S$, and make the curves $C_i$ geodesic. For every pair of pants $P$ in $\mathcal{P}$, every pair of distinct boundary components of $P$ are joined by a unique shortest geodesic arc (the arc we called a seam) that is perpendicular to the boundary components. Using an isotopy, we can deform the chosen arcs such that the arcs coincide with the corresponding seams outside the union of the neighborhoods $N_i$, and such that in every neighborhood $N_i$ they just spin around the cylinder (see \cite[Figure 4.19]{Thurston}). Using the orientation of $C_i$, we can then compute the amount of spinning of each of these arcs, see \cite[Figure 4.20]{Thurston}. For every curve $C_i$, the twist parameter is defined as the difference between the amount of spinning of two of the chosen arcs from the two sides of $C_i$ (again, we need to use the orientation of $C_i$ to choose the order of subtraction). 

  The {\it Fenchel-Nielsen parameters} of $x$ is the collection of pairs 
$\left((l_x(C_i),\theta_x(C_i))\right)_{i\in \mathcal{C}}$, where it is understood that if $C_i$ is homotopic to a boundary component, then there is no associated twist parameter, and instead of a pair $(l_x(C_i),\theta_x(C_i))$, we have a single parameter
$l_x(C_i)$.

 We shall say that two hyperbolic structures $x$ and $y$ are {\it Fenchel-Nielsen-equivalent} relative to the pair of pants decomposition $\mathcal{P}$ if their Fenchel-Nielsen parameters are equal.

Given two hyperbolic metrics $x$ and $y$ on $S$, we define their {\it Fenchel-Nielsen distance} with respect to $\mathcal{P}$ as
   
\begin{equation}\label{def:FND}
{d_{FN}(x,y)=\sup_{i=1,2,\ldots} \max\left(\left\vert \log \frac{l_x(C_i)}{l_y(C_i)}\right\vert, \vert l_x(C_i)\theta_x(C_i)-l_y(C_i)\theta_y(C_i)\vert \right) },
\end{equation}
again with the convention that if $C_i$ is the homotopy class of a boundary component of $S$, then there is no twist parameter to be considered.

There are variations on the definition of the Fenchel-Nielsen distance that have shortcomings, see the discussion in Examples \ref{ex:FN1} and \ref{ex:FN2} below. 
  
Note that strictly speaking, the notation $d_{FN}$ should rather be $d_{\mathcal{P}}$, since the definition depends on the choice of the pair of pants decomposition $\mathcal{P}$. But since this pair of pants decomposition will be fixed throughout the paper, we prefer to use the notation $d_{FN}$, which is more appealing.

Also note that as a function on the space of all homotopy classes of hyperbolic structures on $S$, $d_{FN}$ is not a distance function in the usual sense, because it can take the value infinity. But we shall shortly restrict it to a set of homotopy classes of hyperbolic structures on which $d_{FN}$ is a genuine distance.

It is easy to tell the Fenchel-Nielsen distance between two hyperbolic structures that are on the same orbit under Fenchel-Niesen twists along multicurves whose homotopy classes belong to the collection $\mathcal{C}$. But in general there is no practical way to compute the Fenchel-Nielsen distance between two arbitrary elements of Teichm\"uller space, and one of the problems that we address in this paper is to find estimates for such distances, in terms of other data.

Given two hyperbolic structures $x$ and $y$ on $S$, we say that a homeomorphism $f:(S,x)\to (S,y)$ that is isotopic to the identity is {\it Fenchel-Nielsen bounded} (relatively to $\mathcal{P}$) if $d_{FN}(x,y)$ is finite.

Let $H_0$ be a hyperbolic structure on $S$, which we shall consider as a base hyperbolic structure.
We consider the collection of marked hyperbolic structures $(f,H)$ relative to $H_0$, with the property that the marking
$f:H_0\to H$ is Fenchel-Nielsen bounded with respect $\mathcal{P}$. Given two such marked hyperbolic structures $(f,H)$ and $(f',H')$, we write $(f,H)\sim(f',H')$ if there exists an isometry $f'':H\to H'$ which is homotopic to $f'\circ  f^{-1}$. The relation $\sim$ is an equivalence relation on the set of Fenchel-Nielsen bounded marked hyperbolic surfaces $(f,H)$ based at $(S,H_0)$.

\begin{definition}[Fenchel-Nielsen Teichm\"uller space] The {\it Fenchel-Nielsen Teich\-m\"uller space} with respect to $\mathcal{P}$ and to $H_0$, denoted by $\mathcal{T}_{FN}(H_0)$,  is the space of $\sim$-equivalence classes $[f,H]$ of Fenchel-Nielsen bounded marked hyperbolic structures $(f,H)$.

 The function $d_{FN}$ defined above is clearly a distance function on $\mathcal{T}_{FN}(H_0)$.  The {\it basepoint} of this Teichm\"uller space is the equivalence class $[\mathrm{Id},H_0]$.
\end{definition}

 We shall call the distance $d_{FN}$ on $\mathcal{T}_{FN}(H_0)$ the {\it Fenchel-Nielsen distance} relative to the pair of pants decomposition $\mathcal{P}$.
The map 
$$
\mathcal{T}_{FN}(H_0) \ni x \mapsto {(\log(l_x(C_i)), l_x(C_i)\theta_x(C_i))}_{i \geq 1}\in \ell^\infty $$
 is an isometric bijection between $\mathcal{T}_{FN}(H_0)$ and the sequence space $\ell^\infty$. It follows fom general properties of $l^{\infty}$-norms that the Fenchel-Nielsen distance on $\mathcal{T}_{FN}(H_0)$ is complete.

\section{The quasiconformal dilatation of a Fenchel-Nielsen multi-twist}\label{s:twist}

In this section, we give an upper-bound for the Fenchel-Nielsen distance in terms of quasiconformal distance, between two hyperbolic metrics obtained by a Fenchel-Nielsen multi-twist deformation, that is, a composition of twist deformations along a union of disjoint simple closed curves which belong to the given pair of pants decomposition $\mathcal{P}$, under the hypothesis that the closed curves in $\mathcal{P}$ satisfy un upper-bound condition (Theorem \ref{mixed} below). 
This upper-bound will be used in the study of the Fenchel-Nielsen Teichm\"uller space that we make in Section \ref{s:upper} below.

We start with the case of a simple twist, that is, a Fenchel-Nielsen twist along a connected simple closed curve. In this case, the computations are simpler, and they will guide us for the multi-twist map case.

We need to recall a few facts concerning the dilatation $K(t)$ of a quasiconformal mapping  $\tau_\alpha^t:S\rightarrow S_t$ realizing a Fenchel-Nielsen deformation.
A detailed study of the quasiconformal theory of the Fenchel-Nielsen deformation was done in Wolpert in \cite{Wolpert-twist}.   
 
 We shall use a theorem attributed to Nielsen stating that if $f:R_1\rightarrow R_2$ is a quasiconformal homeomorphism between two hyperbolic surfaces whose universal covers are the hyperbolic plane, then any lift  $\tilde{f}:\mathbb{H}^2\rightarrow \mathbb{H}^2$ of  $f$ extends continuously to a homeomorphism on $S_\infty^1$ and that the restriction of this extension to the limit set of the covering group is invariant under homotopies of $f$ (see \cite{Nag} \S 2.3.3). 

  Given a Hyperbolic surface $S$ and a simple closed geodesic  $\alpha$ on $S$, 
we consider the path $(S_t)_{t\in\mathbb{R}}$ obtained from $S$ through 
 Fenchel-Nielsen deformation along $\alpha$. 

It is convenient to work in the universal cover of the surfaces $S$ and $S_t$. We shall assume in this section that this universal cover is the hyperbolic plane $\mathbb{H}^2$, and we shall mention the modifications that are needed when the universal cover is only a subset of $\mathbb{H}^2$. 

We work in the upper-half plane model of $\mathbb{H}^2$, and we follow the exposition given in \cite{IT}.

  There is a neighborhood of $\alpha$ that is embedded in $S$ (a ``collar neighborhood" of $\alpha$), which is of the form:
$$W_\alpha=\{p\in S \    |  \  d(p,\alpha)\leq \omega_\alpha\}$$
where $\omega_\alpha$ satisfies $$\sinh \omega_\alpha \sinh \frac{l_S(\alpha)}{2}=1.$$ 

We take a Fuchsian group $\Gamma$ for $S$ acting on $\mathbb{H}^2$; that is, a group $\Gamma$ that acts properly discontinuously on $\mathbb{H}^2$ and such that $S=\mathbb{H}^2/\Gamma$. Up to conjugating the group $\Gamma$ by an isometry, we can assume that  the hyperbolic transformation $A(z)=\lambda z$ ($\lambda=\exp(l_S(\alpha)) >1$) belongs to $\Gamma$ and that the axis of $A$ covers the geodesic $\alpha$. 

The collar $W_{\alpha}$ is covered by the region
$$\tilde{W}_\alpha=\{z\in \mathbb{H}^2  \  | \  d(z,i\mathbb{R}^+)\leq \omega_\alpha\},$$
with deck transformation generated by $A$. The quotient $\tilde{W}_\alpha/\langle A \rangle$ is isometrically embedded to $S$ and its image is $W_\alpha$. The set $\tilde{W}_\alpha$ can also be described as
\[\tilde{W}_\alpha=\{z\in \mathbb{H}^2  \  | \  \frac{\pi}{2}-\theta_\alpha < \arg z <  \frac{\pi}{2}+\theta_\alpha\}.\]
From Formula \ref{formula:theta}, we have:
\[\theta_\alpha=2 \arctan (\frac{e^{\omega_\alpha}-1}{e^{\omega_\alpha}+1}).\]

Next, for every $t\in \mathbb{R}$, we define a quasiconformal mapping $q$ of $\mathbb{H}^2$ onto itself by

\begin{equation}\label{equ:twist}
qz)=\begin{cases}
z & \text{if } 0<\theta < \frac{\pi}{2}-\theta_\alpha \\
z\exp\left(t(\frac{\theta-\frac{\pi}{2}+\theta_\alpha}{2\theta_\alpha})\right) & \text{if } \frac{\pi}{2}-\theta_\alpha \leq \theta \leq \frac{\pi}{2}+\theta_\alpha \\
z\exp (t) & \text{if }  \frac{\pi}{2}+\theta_\alpha \leq \theta \leq \pi.
\end{cases}
\end{equation}

        \bigskip
  \begin{figure}[!hbp]
\centering
\psfrag{t}{\small $\theta_{\alpha}$}
\includegraphics[width=.80\linewidth]{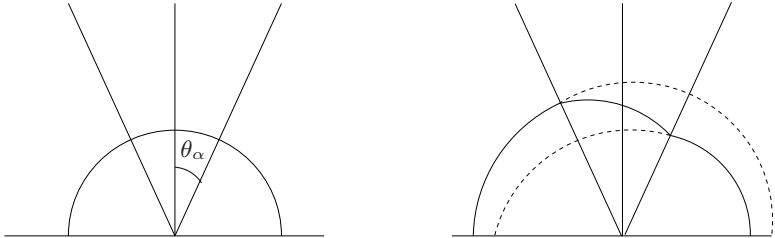}
\caption{\small {The Fenchel-Nielsen deformation.}}
\label{FN}
\end{figure}
\bigskip

 (Note that there is a slight difference between this formula and the one given in  \cite{IT}, because we are using here left twists, and in  \cite{IT} the authors use right twists.)
  From \cite{IT} p. 220, the complex dilatation of $q$ is equal to
$$\mu^t=\frac{i\frac{t}{2\theta_\alpha}}{2-i\frac{t}{2\theta_\alpha}}\chi_I(\theta) \frac{z}{\bar z}, z\in \mathbb{H}^2$$
where $\chi_I$ is the characteristic function of $I=[\frac{\pi}{2}-\theta_\alpha,\frac{\pi}{2}+\theta_\alpha]$ on $\mathbb{R}$. 

Let $\Gamma_\alpha$ be the set consisting of all elements in $\Gamma$ which cover $\alpha$. Then we have \[
\Gamma_\alpha= \{B\circ A \circ B^{-1}  \  | \  B\in \Gamma\}.\]
From here, we can construct a family of quasiconformal self-mappings of $\mathbb{H}^2$  along the lifts of $\alpha$ that induces the Fenchel-Nielsen deformation on $S$. For this purpose, set
$$\mu_\Gamma^t=\sum_{B\in \langle A\rangle\backslash \Gamma}(\mu^t\circ B)\frac{\overline{B'}}{B'}.$$
It follows from the definition that $\mu_\Gamma^t$ is a $\Gamma$-invariant Beltrami differential. The $\mu_\Gamma^t$-quasiconformal mapping of $\mathbb{H}^2$, denoted by $\widetilde{f^t}$, induces a deformation $f^t$ of $S$ which realizes the time-$t$ Fenchel-Nielsen deformation of $S$ along $\alpha$. We normalize each $\widetilde{f^t}$ so that it fixes $0,i,\infty$.

Next we study the action of the quasiconformal map $\widetilde{f^t}$ on the circle at infinity $\mathbb{R}\cup\{\infty\}$. This will be useful to get a lower bound for
the quasiconformal dilatation of $\widetilde{f^t}$.

First notice the following fact, adapted from a reasoning in Kerckhoff's paper \cite{Kerckhoff}  (p. 252).

\begin{lemma}\label{lemma:left} If $\widetilde{f^t}$ fixes $0,i,\infty$, then
$$\widetilde{f^t}(-1)<-e^t \ and \  \widetilde{f^t}(1)< 1.$$
\end{lemma}
\begin{proof}
This follows from the construction of the Fenchel-Nielsen twist deformation which we now describe.

Let $\tilde{\alpha}$ be the lift of the closed geodesic $\alpha$ to the universal cover and let $\gamma$ be the bi-infinite geodesic connecting $1$ and $-1$. 
By assumption, the imaginary axis $i\mathbb{R}^{+}$ is a lift of
 $\tilde{\alpha}$ and intersects $\gamma$
at the point $i$. Under the twist deformation,  $\gamma$ is deformed into a union of arcs $\bar{\gamma}$ 
coming from $\gamma$ under the twist deformation, with endpoints
$\widetilde{f^t}(1)$ and $\widetilde{f^t}(-1)$. See Figure \ref{arcs}.

\begin{figure}[h]
\centering
\includegraphics[scale=0.75, bb=83 188 469 437]{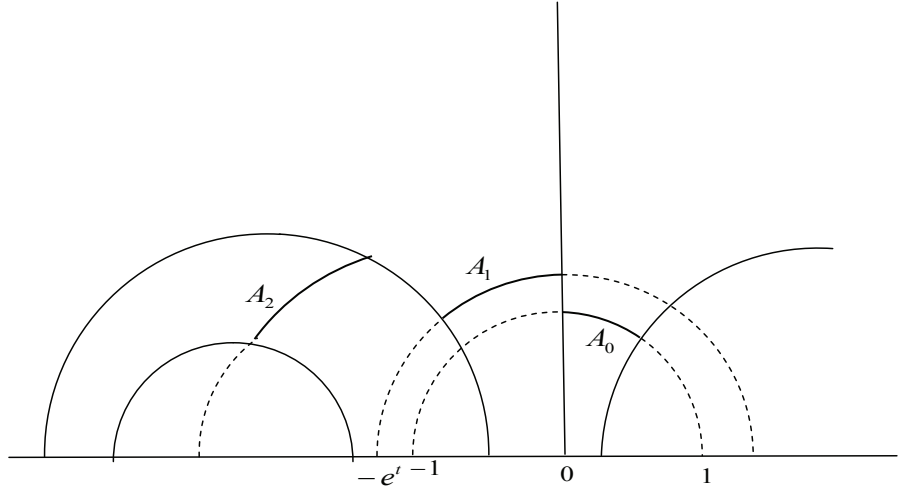}
\caption{\small{The image of $\gamma$ under a simple twist.}}
\label{arcs}
\end{figure}

Note that one such arc $A_0$ passes through the point $i$. If $A_0$ is continued to a bi-infinite geodesic, its endpoints
will be precisely those of $\gamma$. Move along $\bar{\gamma}$ in the left direction, and run along the
 geodesic $i\mathbb{R}^{+}$ by a hyperbolic distance $t$ until coming to the next arc $A_1$. If the new arc is continued in
the forward direction, one of its endpoints is $-e^t$. Similarly, the forward endpoint of the next arc, $A_2$, 
is strictly to the left of $-e^t$. In fact, the forward endpoint of each arc $A_{i+1}$ is strictly to the left of $A_i$. 
Since the forward endpoints of the $A_i^{,}s$ converge to
$\widetilde{f^t}(-1)$, we see that $\widetilde{f^t}(-1)<-e^t$.

The same argument shows that $\widetilde{f^t}(1)< 1$.
\end{proof}

Given four distinct points $a,b,c,d$ on the circle $\mathbb{R}\cup\{\infty\}$, let $H(a,b,c,d)$ be the upper half-plane, considered as a disk, with four distinguished points  $a,b,c,d,$ on its boundary, as in Section \ref{sez:preliminaries-quasiconformal}, and let $\mathrm{mod}(H(a,b,c,d))$ be its conformal modulus. (See Section \ref{sez:preliminaries-quasiconformal} above, for the convention on the choice of two arcs on the boundary of this quadrilateral that is involved in the definition of the modulus.)  Denote
 $\mathrm{mod}(H(\infty,-1,0,e^t))$ by $h(t)$. Then, $h(t)$ is a strictly increasing function and $h(0)=1$, $\lim_{t\rightarrow +\infty}h(t)=\infty$.

  For $0<r<1$, let $\mu(r)$ be the modulus of the Gr\"{o}tzch ring $\mathbb{D}\setminus [0,r]$; that is, the ring domain obtained by deleting the interval $[0,r]$ 
from the unit disk $\mathbb{D}$ (see Figure \ref{domain}).

        \bigskip
  \begin{figure}[!hbp]
\centering
\psfrag{0}{\small $0$}
\psfrag{r}{\small $r$} 
\includegraphics[width=.4\linewidth]{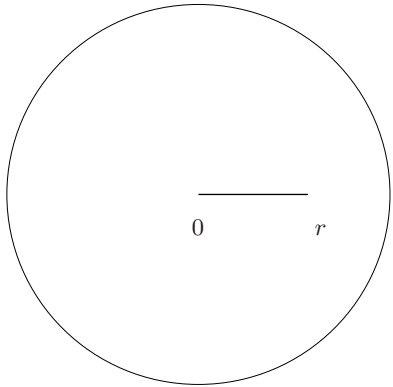}
\caption{\small{The Gr\"{o}tzch domain. }}
\label{domain}
\end{figure}
\bigskip

 Then $h(t)$ is related to $\mu(r)$ by the following equality:
$$
h(t)=\frac{2}{\pi}\mu( {\sqrt{\frac{1}{1+\lambda}} }), \ \mathrm{where } \ \lambda=e^t
$$
(see \cite{Lehto}, p. 60-61).
The function $\mu(r)$ has been systematically studied. The following lower bound of $\mu(r)$ is given in \cite{Lehto}, p. 61.
$$\mu(r)> \frac{2}{\pi} \log \frac{{(1+ \sqrt{1-r^2})^2}}{r}.$$
We will also use below the following formula for $0<r<1$:
$$\mu'(r)=-\frac{\pi^2}{4r(1-r^2)K(r)^2}$$
where 
$$K(r)=\int_{0}^{1} \frac{dx}{\sqrt{(1-x^2)(1-r^2x^2)}}.$$
For the proof, see p. 82 $(5.9)$ of \cite{AVV}.

It follows from the above formula for the function $h$ that 
$$h'(t)=\frac{-1}{\pi}\mu'( {\sqrt{\frac{1}{1+\lambda}} })\frac{1}{\sqrt{(1+\lambda)^3}}> 0,$$
with $\lambda=e^t$ as before.
\begin{lemma} \label{lemma:K-lower-bound} For all $t>0$, we have
$K(f^t)\geq h(t)$.
\end{lemma}
\begin{proof}
It follows from the geometric definition
of quasiconformal map that
\[K(\widetilde{f^t}) \geq \frac{\mathrm{mod}\left(H(\widetilde{f^t}(-1),\widetilde{f^t}(0),\widetilde{f^t}(1),\widetilde{f^t}(\infty))\right)}{\mathrm{mod}(H(-1,0,1,\infty))}.\]
Since $\mathrm{mod}(H(-1,0,1,\infty))=1$ and $\widetilde{f^t}$ fixes $0,\infty$, we have
\[K(\widetilde{f^t}) \geq \mathrm{mod}\left(H(\widetilde{f^t}(-1),0,\widetilde{f^t}(1),\infty)\right).\]
We have shown in Lemma \ref{lemma:left} that  $\widetilde{f^t}(-1)< -e^t$ and $\widetilde{f^t}(1)<1$. If we let
$\mathcal{F}$ be the family of curves joining the $a$-sides of $H(-e^t,0,1,\infty)$, and $\mathcal{F'}$
be the family of curves joining the $a$-sides of $H(\widetilde{f^t}(-1),0,\widetilde{f^t}(1),\infty)$, then
we obviously have $\mathcal{F}\subseteq \mathcal{F'}$. It then follows from the definition of extremal length that
$$\mathrm{Ext}(\mathcal{F})\geq \mathrm{Ext}(\mathcal{F'}).$$
By the relation of extremal length and modulus, we have
\[\mathrm{mod}\left(H(\widetilde{f^t}(-1),0,\widetilde{f^t}(1),\infty)\right) \geq  \mathrm{mod}(H(-e^t,0,1,\infty)).\]
Then
$$K(\widetilde{f^t}) \geq \mathrm{mod}(H(-e^t,0,1,\infty)).$$
Note that $\mathrm{mod}(H(-e^t,0,1,\infty))= \mathrm{mod}(H(\infty,-1,0,e^t)$ and we get
\begin{equation}\label{inequ:1}
K(\widetilde{f^t})\geq \mathrm{mod}(H(\infty,-1,0,e^t)=h(t).
\end{equation}
\end{proof}

\begin{lemma} \label{lem:qcc} If $K(f^t)\leq L$, then there is a constant $\delta>0$ depending on $L$ such that
$$t\leq \delta \log K(f^t).$$
\end{lemma}
\begin{proof}

Note that the function $h(x)$ is increasing and $h(x)\to \infty$ as $x\to\infty$. As a result, there is a positive constant $T$ depending on $L$ such that $t\leq T$. Since $h'(x)>0$ for all $x>0$ and since the function $h'(x)$ is continuous, there exists a  positive 
constant $D$ depending on $T$ such that
$$h'(x)\geq D, \ \mathrm{for} \ \mathrm{all} \  t\in [0,T].$$
This means that
\begin{equation}\label{inequ:2}
h(x)\geq 1+Dx\geq e^{M x},
\end{equation}
for all $x\in [0,T]$, where $M$ is some small constant depending on $T$.

Combining Lemma \ref{lemma:K-lower-bound} with $(\ref{inequ:2})$, we have $\log K(f^t)\geq M t$. By setting $\delta=\frac{1}{M}$, we obtain
$K(f^t)\geq Mt$. By setting $\delta=1/M$, we obtain 
\[t\leq \delta \log(K^t).\]

 \end{proof}
It follows from the proof of the above lemma that we have 
\begin{lemma}\label{lem:6}
If $0\leq t< T$, there exists a constant $M$ depending on $T$ such that \[Mt\leq \log K(f^t).\]
\end{lemma}
 
From Lemma \ref{lemma:K-lower-bound}, we have $\log K(f^t)\geq \log h(t)$ and then $\log K(f^t)\geq \delta t$. 
 
Now we need to consider the case of a multi-twist; that is, the case of a composition of Dehn twists along the collection $\{C_i\}$ of disjoint curves.

        \bigskip
  \begin{figure}[!hbp]
\centering
\includegraphics[width=.40\linewidth]{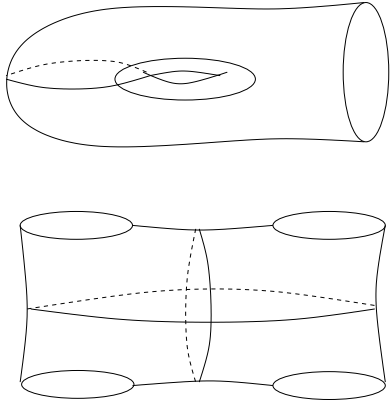}
\caption{\small {The curve $\beta_i$ used in the proof of Lemma \ref{lem:angle}. In each case, we have represented the simple closed curves $C_i$ and $\beta_i$.}}
\label{dual}
\end{figure}
\bigskip

For this, we need to
choose a convenient collection of simple closed curve $\{\beta_i\}$ in $S$.
This is given in the following lemma:
\begin{lemma}\label{lem:angle}
For each $i=1,2,\ldots$, we can find a simple closed curve $\beta_i$ with the following peroperties :
\begin{enumerate}
\item $\beta_i$ intersects  $C_i$in a minimal number of points (that is, on one or in two points):
\item $\beta_i$ does not intersect the $C_j$ for $j\not=i$;
\item the angle (or the two angles) that $\beta_i$  makes at its intersection with $C_i$ is bounded from below by a positive constant that does not depend on $i$. 
\end{enumerate}
\end{lemma}

\begin{proof}
Consider the surface $S'$ cut along the collection of closed geodesics $\{C_i\}$. This surface is a disjoint union of generalized pairs of pants, whose lengths of boundary geodesic components are, by assumption, bounded above by a constant that does not depend on $i$. For a given $i$, consider the generalized pairs of pants  in the decomposition that have a geodesic arising from the curve $C_i$ on their boundary (there are one or two such pairs of pants). To simplify notation, we shall also call $C_i$ such a geodesic that arises from $C_i$.
In each such generalized pair of pants, we consider the geodesic arc of shortest length that joins $C_i$ to itself and that is not homotopic to a point. There are one or two such geodesic arcs associated to $C_i$, depending on the number of generalized pairs of pants containing $C_i$ on their boundary. For simplicity, we assume that there is only one such geodesic arc, and we call it $\gamma_i$; the case where there are two arcs can be dealt with in the same manner. 

We choose the geodesic $\beta_i$ on $S$ in such a way that it is homotopic to the curve is made out of the union of $\gamma_i$ with a geodesic segment $\gamma'_i$ contained in the geodesic boundary curve $C_i$, 
see Figure \ref{dual}. 
From the Collar Lemma (Lemma  \ref{lemma:collar}), the length of $\gamma_i$ is bounded from below by a constant that does not depend on $i$. Furthermore, the length of the segment $\gamma'_i$ is bounded from above by  constant that does not depend on $i$.  

Let $\phi$ be one of the two angles that the closed geodesic $\beta_i$ makes with the closed geodesic $C_i$. Recall from Lemma \ref{lemma:collar} that the closed geodesic $C_i$ has a collar neighborhood whose width is bounded below by a constant $2d$ that only depends on the length $l(C_i)$ of the geodesic $l(C_i)$. recall also that if $l(C_i)$ is bounded above by $M$, then  $d$ is bounded below by a constant that only depends on $M$. 

We make estimates in the upper-half plane. We assume that a lift of the closed geodesic $C_i$ is the imaginary axis, and that a lift of the closed geodesic $\beta_i$ intersects this axis at the point $i$ in the complex plane. 
We let $\mathcal{C}$ be the Euclidean circle containing this hyperbolic geodesic that covers $\beta_i$, and we let $c$ be its Euclidean centre on the $x$-axis, and $i^*$ and $A^*$ the two endpoints of that geodesic on the real axis, such that $i^*$, $i$, $A$ and $A^*$ are in that order on this geodesic (see Figure \ref{circle}). We have, from Formula (\ref{formula:theta}):

         \bigskip
  \begin{figure}[!hbp]
\centering
\psfrag{f}{\small $\phi$}
\psfrag{t}{\small $\theta$}
\psfrag{A}{\small $A$}
\psfrag{B}{\small $A^*$}
\psfrag{i}{\small $i$}
\psfrag{j}{\small $i^*$}
\psfrag{c}{\small $c$}
\psfrag{6}{\small $b_2$}
\psfrag{7}{\small $h_1$}
\includegraphics[width=.50\linewidth]{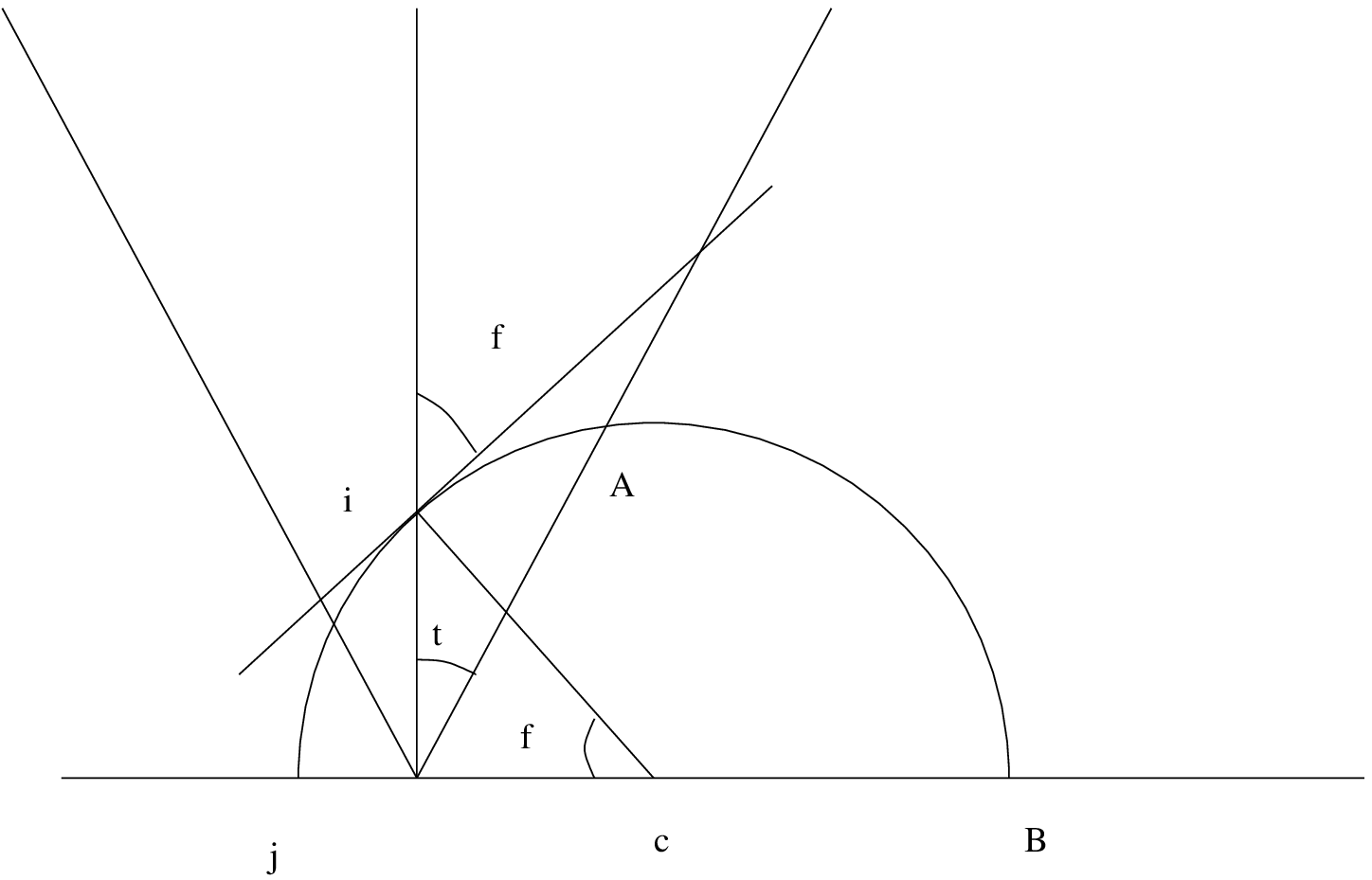}
\caption{\small {}}
\label{circle}
\end{figure}
\bigskip

\[\theta=2\arctan \left(\frac{e^d-1}{e^d+1}\right).\]
Note that $0<d<\infty$, which implies $0<\frac{e^d-1}{e^d+1}<1$, for $0<\theta<\frac{\pi}{2}$. 

We have $c=\cot \phi$, and therefore $\phi\to\infty$ as $c\to\infty$.

The equation of the circle $\mathcal{C}$, in  cartesian  coordinates, is  $(x-c)^2+y^2=r^2$, with $r^2=c^2+1$. Thus, the equation is
\[x^2+y^2-2cx -1=0.\]

The Euclidean coordinates of the point $A$ are $(\lambda \sin \theta, \lambda \cos \theta)$, where $\lambda>0$ is the Euclidean distance from $A$ to the origin.
We can find the value of $\lambda$  by replacing, in the equation of the  circle $\mathcal{C}$, $x$ and $y$ by the coordinates of $A$. We obtain:
\[\lambda = c\sin\theta + \sqrt{c^2\sin^2\theta +1}.\]

We compute the hyperbolic distance $d(i,A)$ from $i$ to $A$ using the formula
\[
d(i,A)=\frac{(i-A^*)(A-i^*)}{(A-A^*)(i-i^*)}
\]
in which $A=\lambda\sin\theta + i\lambda \cos \theta$, $i^*=c-r=c-\sqrt{1+c^2}$ and $A^* = c+r = c+\sqrt{1+c^2}$. 

A computation gives
\begin{equation}\label{eq:distA}
d(i,A)^2= \frac{(c+\sqrt{1+c^2})^2+1}{(-c+\sqrt{1+c^2})^2+1}\cdot
 \frac{(\lambda\sin\theta -c+\sqrt{1+c^2})^2+\lambda^2\cos^2\theta}
 {(-\lambda\sin\theta +c+\sqrt{1+c^2})^2+\lambda^2\cos^2\theta}.
\end{equation}

For $x>0$, we have $x\leq \sqrt{x^2+1}\leq x+1$. 

This gives
$(c+\sqrt{1+c^2})^2+1\leq 2$ and $(-c+\sqrt{1+c^2})^2+1\geq 4c^2$.

Thus, we have
\[
\frac{(c+\sqrt{1+c^2})^2+1}{(-c+\sqrt{1+c^2})^2+1}\geq 2c^2.
\]

On the other hand, we have 
 \begin{eqnarray*}
 -\lambda\sin\theta +c+\sqrt{1+c^2}&= &- (c\sin^2\theta + \sqrt{\sin^2\theta (1+c^2\sin^2\theta)}
 +c+\sqrt{1+c^2}\\
 &=& c(1-\sin^2\theta) + \sqrt{1+c^2}-\sqrt{\sin^2\theta +c^2 \sin^4\theta}.
\end{eqnarray*}

From the concaveness of the function $x\mapsto\sqrt{x}$, we have, for $x$ and $y>0$,
\[\sqrt{x}-\sqrt{y}\leq (x-y)\frac{1}{2}\frac{1}{\sqrt{y}}.\]

Using this inequality, we obtain

 \begin{eqnarray*}
&c(1-\sin^2\theta) + \sqrt{1+c^2}-\sqrt{\sin^2\theta +c^2 \sin^4\theta}\leq\\
&\leq  
c(1-\sin^2\theta) + (1+c^2 -\sin^2\theta -c^2\sin^4\theta)\frac{1}{2}\frac{1}{\sqrt{\sin^2\theta +c^2\sin^4\theta}}=\\
&= c(1-\sin^2\theta) + \frac{(1-\sin^2\theta)+c^2(1-\sin^2\theta)}{2\sqrt{\sin^2\theta +c^2\sin^4\theta}}\leq \\
&\leq c(1-\sin^4\theta) 
+ \frac{1}{2c\sin^2\theta}(1-\sin^4\theta)
+\frac{c}{2\sin^2\theta}(1-\sin^4\theta)=\\
&=(1-\sin^4\theta) \left( c+\frac{c}{2\sin^2\theta}+\frac{1}{2c\sin^2\theta}\right)\leq \\
&\leq \frac{3}{2}\frac{c}{\sin^2\theta}(1-\sin^4\theta).
 \end{eqnarray*}
 
 We also have 
  \begin{eqnarray*}
 & -\lambda \sin\theta +c+\sqrt{1+c^2} + \lambda^2 \cos^2\theta \leq\\
  &\leq   \frac{9}{4}\frac{c^2}{\sin^4\theta}(1-\sin^4\theta)^2 + (2 c\sin\theta +1)^2 (1-\sin^2\theta)\leq \\
  &\leq 
    \frac{9}{4}\frac{c^2}{\sin^4\theta}(1-\sin^4\theta)^2 + (2 c\sin\theta +1)^2 (1-\sin^4\theta)=\\
    &= (1-\sin^4\theta)\left( \frac{9}{4}\frac{c^2}{\sin^4\theta}(1-\sin^4\theta)+ (2c\sin\theta +1)^2 \right)\leq \\
    &\leq 
    (1-\sin^4\theta)\left( \frac{9}{4}\frac{c^2}{\sin^4\theta}+9c^2 \right)\leq\\
    &\leq  (1-\sin^4\theta)\cdot 12 \frac{c^2}{\sin^4\theta}
   \end{eqnarray*}
   
   and 
  \[
(\lambda \sin\theta -c + \sqrt{1+c^2})^2 +\lambda^2\cos^2\theta\\
   \geq \lambda \sin^2+\lambda^2\cos^4\theta\\
 =\lambda^2\\
  \geq 4c^2\sin^2\theta.
    \]
    
     Replacing in (\ref{eq:distA}), we get
     \[d(i,A)^2\geq 2c^2\frac{4c^2\sin^2\theta}{(1-\sin^4\theta) \cdot 12 \frac{c^2}{\sin^4\theta}}= \frac{2}{3} \frac{c^2\sin^6\theta}{(1-\sin^2\theta)}.\]
     
     Setting $p=\frac{e^d-1}{e^d+1}$, we have $\theta = 2\arctan P$ and
      \begin{eqnarray*}
      \sin\theta &=& 2 \sin(\arctan P)\cos(\arctan P)\\
      &=& 2  \frac{p}{\sqrt{1+p^2}}\frac{1}{\sqrt{1+p^2}}\\
      &=&\frac{2p}{1+p^2}.
          \end{eqnarray*}
     Therefore, 
     \[\sin \theta = \frac{e^{2d}-1}{e^{2d}+1}\]
     and 
    \begin{eqnarray*}
    1-\sin^4\theta &=& 1-\left(  \frac{e^{2d}-1}{e^{2d}+1}\right)^4 =\frac{(e^{2d}+1)^4 - (e^{2d}-1)^4}{(e^{2d}+1)^4}\\&=& \frac{4e^{6d}+4e^{2d}}{(e^{2d}+1)^4}\leq \frac{8e^{6d}}{e^{8d}}=\frac{8}{e^{2d}}.
     \end{eqnarray*}
     
     We obtain
     \[d(i,A)^2\geq \frac{1}{12}c^2 \left( \frac{e^{2d}-1}{e^{2d}+1}\right)^6 e^{2d}\]
     or, equivalently,
    \begin{equation}\label{eq:dist}
    d(i,A)\geq \frac{\sqrt{3}}{6}\cot\phi \left( \frac{e^{2d}-1}{e^{2d}+1}\right)^3e^d.
    \end{equation}
     
    The length of the closed geodesic $\beta_i$ is bounded below by $M+4d$. From Inequality (\ref{eq:dist}), we have
    \[\frac{\sqrt{3}}{6}\cot \phi \left( \frac{e^{2d}-1}{e^{2d}+1}\right)e^d\leq M+4d,\]
  which gives
  
\begin{equation}\label{cot} \cot\phi \leq \frac{M+4d}{e^d}\frac{e^{2d}-1}{e^{2d}+1}.
   \end{equation}
      
 From Lemma \ref{lemma:collar},  $d$ is bounded from below by a constant that only depends on $M$. Therefore, the right hand side in (\ref{cot}) bounded from above by a constant that only depends on $M$. Thus, $\phi$ is bounded from below by a constant that only depends on $M$.
This proves the lemma.
\end{proof}

In the rest of this section, we prove the following theorem.
\begin{theorem}\label{mixed}
Let $S$ be a hyperbolic surface with a pair of pants decomposition $\mathcal{P}=\{C_i\ \vert \ i=1,2,\ldots\}$ such that $l_S(C_i)\leq L_0$ for all $i=,2,\ldots$.
Let $t=\{t_i\ \vert \ i=1,2,\ldots\}$ be a sequence of positive real numbers and let $S_t$ be the hyperbolic metric obtained by a Fenchel-Nielsen multi-twist along $C_i$, of distance $t_i$ measured on $C_i$, for each $i$.
Then if $d_{qc}(S,S_t)< T_0$, we have
$$d_{FN} (S,S_t)\leq C d_{qc}(S,S_t)$$
where $C$ is a positive constant depending on $L_0$ and $T_0$.
\end{theorem}

\begin{proof}
Recall that $d_{FN}(S,S_t)=\sup_i\{|\log  \displaystyle \frac{l_{S_t}(C_i)}{l_{S}(C_i)}|,|t_i|\}$, and that, by Wolpert's inequality, we have
$\sup_i\{|\log \displaystyle \frac{l_{S_t}(C_i)}{l_{S}(C_i)}|\}\leq d_{qc}(S,S_t)$. To prove the theorem, it suffices to show that there exists a constant $C$ that depends only on $L_0$ and $T_0$ such that
$|t_i|\leq Cd_{qc}(S,S_t)$ for all $i=,2,\ldots$.

We first make a reduction. Given a quasiconformal homeomorphism $q$ between $S$ and $S^t$, we start by extending it to a homeomorphism between the complete extensions of the hyperbolic surfaces $S$ and $S^t$ without changing the quasiconformal constant of $q$. In fact, the complete extension of $S$ (or $S^t$) coincides with the Nielsen extension of the conformal structure underlying $S$ (respectively $S^t$) as defined by Bers (see \cite{Bers1976}). (Note that we have defined the complete extension of a hyperbolic surface, whereas the Nielsen extension is defined for a conformal structure. In fact, the Nielsen extension of a conformal surface is the complete extension of its associated intrinsic metric.)
Since the hyperbolic metrics $S$ and $S^t$ are upper-bounded, they are complete, and in this case taking their Nielsen extensions are obtained by in adding funnels. (In the general case, besides funnels, half-disks may be needed, see \cite{Basmajian}.)
Now by replacing the hyperbolic metrics $S$ and $S^t$ by their complete extensions, we may assume that the universal covering of these surfaces is the hyperbolic plane $\mathbb{H}^2$.

With this assumption, we let $\widetilde{f^t}:\mathbb{H}^2 \to \mathbb{H}^2$ be a lift of a quasiconformal map between $f^t:S\to S^t$ that realizes the Fenchel-Nielsen multi-twist deformation. 
We note that the restriction of $\widetilde{f^t}$ on the limit set of the covering group transformations only depends on the homotopy class of the map $f^t$. 
We normalize
each $\widetilde{f^t}$ such that it fixes $0,i,\infty$.

 In the case where $S_t$ is obtained from $S$ by the Fenchel-Nielsen twist along a single  closed geodesic $C_i$, we already proved the theorem in three steps, which we briefly recall for reference use below.
 
\noindent {\it Step 1.}
Denoting $C_i$ by $\alpha$, we assumed that  $\alpha$ lifts to the axis $i\mathbb{R}^+$, and we considered the infinite geodesic $\gamma$  geodesic connecting $-1$ and $1$ in the upper-half space model. This geodesic makes a right angle with the chosen lift of $\alpha$. The image $\widetilde{f^t}(\gamma)$ is an infinite arc connecting $\widetilde{f^t}(-1)$
and $\widetilde{f^t}(1)$. Then we showed (Lemma \ref{lemma:left})  that
\begin{equation}\label{equ:mix1}
\widetilde{f^t}(-1)<-e^t \ and \  \widetilde{f^t}(1)< 1.
\end{equation}

\noindent {\it Step 2.} Let $h(t)$ be the conformal modulus of the quadrilateral $H(\infty,-e^t,0,1)$. We used (\ref{equ:mix1})
to obtain (Lemma \ref{lemma:K-lower-bound})
\begin{equation}\label{equ:mix2}
K(\widetilde{f^t})\geq h(t).
\end{equation}
\noindent {\it Step 3.} We used the fact that  $h(0)=1$, that $h(t)$ is strictly increasing, and that $h'(t)>0$ to show that
 that when $t$ is bounded, we have (see Formula (\ref{inequ:2}) above):
\begin{equation}\label{equ:mix3}
h(t)\geq e^{Ct}.
\end{equation}

Since the lift of the Teichm\"uller extremal map from $S$ to $S^t$ has the same boundary extension  as that of $\widetilde{f^t}$,  (\ref{equ:mix2}) and (\ref{equ:mix3}) combined  prove the theorem in the case of a simple Fenchel-Nielsen twist.

Now we deal with the general case; that is, the case where
 $S^t$ is obtained by the Fenchel-Nielsen multi-twist deformation along the collection of curves $\{C_i\}$. In this case, the geodesic arc $\gamma$ perpendicular to the lift of  $C_i$ that we
have chosen above may intersect other curves in the collection $\{C_i\}$, and the estimate we did in {Step 1 }does not hold. In particular, $\gamma$ may be deformed to the left by the twist along a curve $C_i$ and to the right by some twist along another curve $C_j$, and so on.

To each curve $C_i$, we associate the curve $\beta_i$ provided by Lemma \ref{lem:angle}. 
We assume as in the above special case that the axis $i\mathbb{R^{+}}$ is a lift of $C_i$, and that there is a lift  $\tilde{\beta_i}$   of $\beta_i$ which intersects $i\mathbb{R^{+}}$ at the point $i$. Let $x_1<0<x_2$ be the two endpoints of $\tilde{\beta_i}$. Then we have, by the same argument used in the special case of a simple Fenchel-Nielsen twist:
\begin{equation}\label{equ:mix4}
\widetilde{f^t}(x_1)<De^{t_i}x_1, \widetilde{f^t}(x_2)<x_2 \ \mathrm{for} \ \mathrm{all} \ t_i>0,
\end{equation}
where $D$ is a constant that does not depend on $i$.

It follows from the geometric definition of a quasiconformal map that
\[K(\widetilde{f^t}) \geq \frac{\mathrm{mod}\left(H(\widetilde{f^t}(x_1),\widetilde{f^t}(0),\widetilde{f^t}(x_2),\widetilde{f^t}(\infty))\right)}{\mathrm{mod}(H(x_1,0,x_2,\infty)}.\]
By $(\ref{equ:mix4})$ and the monotony property of modulus, we have
$$\mathrm{mod}(H(\widetilde{f^t}(x_1),0,\widetilde{f^t}(x_2),\infty)) \geq  \mathrm{mod}(H(De^{t_i}x_1,0,x_2,\infty)).$$
Now we note that
\begin{eqnarray*}
\mathrm{mod}(H(De^{t_i}x_1,0,x_2,\infty))&=&\mathrm{mod}(H(\infty,-x_2,0,-De^{t_i}x_1)) \\
&=&\mathrm{mod}(H(\infty,-1,0,D|\frac{x_1}{x_2}|e^{t_i})).
\end{eqnarray*}
As a result,
$$K(\widetilde{f^t})\geq \frac{\mathrm{mod}(H(\infty,-1,0,D|\frac{x_1}{x_2}|e^{t_i}))}{\mathrm{mod}(H(\infty,-1,0,\frac{|x_1|}{x_2}))}=g(t_i).$$

The function $g(t_i)$ has  properties similar to those of $h(t_i)$, that is, $g(0)=1$,
$g(t_i)$ is strictly increasing and $\lim_{t_i\to +\infty}g(t_i)=\infty$. Moreover, $g'(t_i)>0$ for all $t_i>0$. Finally, by the same arguments used
in the final step of the simple twist case, we have $K(\widetilde{f^t})\leq C|t_i|$ where $C$ is a constant depending on $L_0$ and $T_0$.
  \end{proof}

\section{Hyperbolic structures with an upper-bound condition}\label{s:upper}

We shall say that a hyperbolic metric $H$ on $S$ {\it satisfies an upper-bound condition}, or that it is {\it upper-bounded} (with respect to $\mathcal{P}$) if the following holds:

 \begin{equation}\label{upper} \exists M>0,\forall i=1,2,\ldots, l_H(C_i)\leq M.
 \end{equation}
where $\mathcal{C}=\{C_i\}_{i=1}^{\infty}$ is as above  the collection of homotopy classes of simple closed curves associated to $\mathcal{P}$.

Note that we already made such an assumption in the hypothesis of Theorem \ref{mixed}.

The definition, for a hyperbolic structure $H$, of being upper-bounded, depends on the choice of the pair of pants decomposition $\mathcal{P}$. There are some hyperbolic surfaces that are not upper-bounded with reference to a certain pair of pants decomposition, but that are upper-bounded with reference to another one (see Example \ref{ex:pants1} below). Furthermore, there exist hyperbolic structures that are not upper-bounded with reference to any pair of pants decomposition  (see Example \ref{ex:pants2} below).

We start with the following 
\begin{lemma}
If a hyperbolic metric $H_0$ on $S$ is upper-bounded with respect to $\mathcal{P}$, then, every element of $\mathcal{T}_{FN}(H_0)$ (respectively $\mathcal{T}_{qc}(H_0)$) is also upper-bound with respect to $\mathcal{P}$ (but not necessarily with the same constant $M$).  

\end{lemma}

\begin{proof}
The proof of the lemma in the case of $H\in \mathcal{T}_{FN}(H_0)$ is immediate from the definitions. For the case where $H\in \mathcal{T}_{qc}(H_0)$, by Wolpert's inequality (Theorem \ref{theorem:Wolpert}) if $f:(S,H_0)\to (S,H)$ is a $K$-quasiconformal homeomorphism, then for every $C_i$ in $\mathcal{C}$, we have $l_H(f(C_i))\leq K l_{H_{0}}(C_i)$, which shows that $H$ also satisfies an upper-bound condition.
 
 \end{proof}
 
The main result of this section is Theorem \ref{th:bilipschitz} below saying that if $H_0$ is an upper-bounded hyperbolic structure, then we have a set-theoretic equality $\mathcal{T}_{qc}(H_0)= \mathcal{T}_{FN}(H_0)$ and that, furthermore, the identity map between the two metric spaces $(\mathcal{T}_{qc}(H_0), d_{qc})$ and $(\mathcal{T}_{FN}(H_0),d_{FN})$ is a locally bi-Lipschitz homeomorphism. 
 
 We start with a few lemmas that will be useful in the proof of this result.

\begin{lemma}     \label{lemma:difflengths}
Let $(f,H)$ and $(f',H')$ be two marked hyperbolic structures on $S$ with Fenchel-Nielsen coordinates
$c=(l_i,\theta_i)_{i\geq 1}$ and $c'=(l_i',\theta_i)_{i\geq 1}$ respectively. (Thus, we assume all the twist parameters are the same for the two structures). Assume that for some constant $N$, we have $l_i,l_i' \leq N$ for all $i$. If $d_{FN}(c,c')$ is finite, then there is a quasiconformal mapping $q:H  \to H'$ such that $f'^{-1} \circ q \circ f$ is homotopic to the identity, and 
$$\log K(q) \leq 3 C(N) d_{FN}(c,c').$$ 
\end{lemma}

\begin{proof}
We may assume that the curves $C_k$ are geodesic for both structures $H$ and $H'$. By Theorem \ref{theorem:Bishop}, for every pair of pants $P$ belonging to the decomposition $\mathcal{P}$, we can construct a quasiconformal mapping $q_P$ from $P$ equipped with the structure induced by $H$ to $P$ equipped with the structure induced by $H'$, and such that the quasiconformal dilatation of $q_P$ satisfies $\log K(q_P) \leq 3 C(N) d_{FN}(c,c')$.

We now need to show that for every boundary curve $C_k$ of the pairs of pants $P$, the two maps $q_P, q_{P'}$ defined on the two sides of $C_k$ agree on $C_k$. To do this,  note that the restriction of each of the two maps to $C_k$ simply rescales the distances by the same factor, hence we only need to check that the maps send one given point to the same point. Using the notations of Theorem \ref{theorem:Bishop}, consider one of the point $p_{i,j}$ (with $i \neq j$) lying on the curve $C_k$. As the twist parameters are the same for $c$ and $c'$, these points go to the same point.

This shows that the set of quasiconformal mappings $q_P$ gives rise to a homeomorphism $q$ defined on the whole surface $S$. 

To see that this map is quasiconformal, consider small coordinate patches contained in one or two pairs of pants. For a coordinate patch $U$ contained in only one pair of pants $P$, we know that $q$ agrees with $q_P$, hence its quasiconformal dilatation there satisfies $\log K(q|_U) \leq 3 C(N) d_{FN}(c,c')$. If the coordinate patch $U$ is contained in two pairs of pants $P,P'$ separated by a curve $C_k$, consider the Beltrami differentials of the two maps $q_P, q_{P'}$. As the curve $C_k$ has measure zero, the value of the Beltrami differentials at the points on that curve can be neglected. Thus, we have a new Beltrami differential on $U$ whose essential supremum is the supremum over all the essential suprema of Beltrami differentials on the various pairs of pants. Hence, again we have $\log K(q|_U) \leq 3 C(N) d_{FN}(c,c')$. 

This proves that $q$ is quasiconformal with logarithmic dilatation bounded by $3 C(N) d_{FN}(c,c')$.    
 
\end{proof}

\begin{lemma}\label{lem:iineq}
Let $(f,H)$ and $(f',H')$ be two marked hyperbolic structures on $S$, and
let $c = (l_i,\theta_i)$ and $c' = (l_i,\theta_i')$ be the corresponding Fenchel-Nielsen coordinates; that is, we are assuming that $c$ and $c'$
have  the same length but different twist parameters. Suppose that for some positive constant, $l_i \leq N$, for all $i=,2,\ldots$. If $d_{FN}(c,c')$ is finite, then there exists a quasiconformal mapping $q:H \to H'$ such that $f'^{-1} \circ q \circ f$ is homotopic to the identity and 
$$\log K(q) \leq \frac{1}{L(N)} d_{FN}(c,c') \sqrt{1 + \frac{\big(d_{FN}(c,c')\big)^2}{16(L(N))^2}},$$ 
where $L(N) = 2 \arctan(2 e^N)$.
\end{lemma}
 \begin{proof}
By Lemma \ref{lemma:collar} every simple closed curve $C_i$ has a geometric  annular regular neighborhood $N(C_i)$ of width $2 B(N)$. By Lemma \ref{lemma:conformalcylinder}, the bounded hyperbolic annulus $N(C_i)$ is conformally equivalent to a Euclidean cylinder $C \times I$ with $C$ a circle of length $l_i$ and $I$ an interval of length $2 L(N)$ (say $I = [0,2L(N)]$), where
$$L(N) = 2 \arctan\left(\frac{e^{B(N)}-1}{e^{B(N)}+1}\right) = 2 \arctan(2 e^N)$$
For each $i=1,2,\ldots$, the Fenchel-Nielsen twist of angle $\alpha_i = \theta_i' - \theta_i$ is induced by an affine map on the universal covering $\mathbb{R} \times [0,2 L(N)] \to \mathbb{R} \times [0,2 L(N)]$. Explicitly, the twist is induced by 
$$\phi_{\alpha_{i}} : \mathbb{R} \times [0,2 L(N)] \ni (x,y) \mapsto   \left(x + A y, y\right) \in \mathbb{R} \times [0,2 L(N)]$$
with $A = \frac{\alpha_{i} l_i}{2L(N)}$
By Lemma \ref{lemma:affinequasiconformal}, the quasiconformal dilatation of this affine map is 
$$K(\phi_{\alpha_{i}}) = 1 + \frac{1}{2}A^2 + \frac{1}{2}|A| \sqrt{4+A^2} \leq 1 + |A| \sqrt{4+A^2} $$
and 
\[\log K(\phi_{\alpha_{i}}) \leq |A| \sqrt{4+A^2} = \frac{\alpha l_i}{2L(N)} \sqrt{4+ \left(\frac{\alpha l_i}{2L(N)}\right)^2}.\]

To construct the map $q$ we proceed in the following way. Denote by $P_{ijk}$ (resp. 
$P_{ijk}'$) the pair of pants of $H$ (resp. $H'$) bounded by the curves $C_i, C_j, C_k$, 
and denote by $q_{ijk}:P_{ijk} \mapsto P_{ijk}'$ the isometry between them preserving 
$C_i, C_j, C_k$. The value $q(x)$ is defined as $q_{ijk}(x)$ if $x$ is in $P_{ijk}$ but 
not in the union of the neighborhoods $N(C_i), N(C_j), N(C_k)$, and it is defined as 
$\phi_{\alpha_i}(x)$ if $x$ is in the neighborhood $N(C_i)$. This is a homeoomorphism, it 
is conformal outside of the union of the neighborhoods $N(C_i)$, and on those 
neighborhoods it has quasiconformal dilatation bounded as required by the statement.

\end{proof}

From Lemmas \ref{lemma:difflengths} and  \ref{lem:iineq} we obtain the following:

\begin{proposition}\label{prop:iineq} Let $(f,H)$ and $(f',H')$ be two marked hyperbolic structures on $S$
with Fenchel-Nielsen coordinates $c = (l_i,\theta_i)$ and $c' = (l_i',\theta_i')$ respectively,  such that, for some constant $N$, $l_i,l_i' \leq N$. If $d_{FN}(c,c')$ is finite Then there is a quasiconformal mapping $q:H \to H'$ such that $f'^{-1} \circ q \circ f$ is homotopic to the identity and 
$$\log K(q) \leq  d_{FN}(c,c') \left[ 3 C(N) + \frac{1}{L(N)} \sqrt{1 + \frac{\big(d_{FN}(c,c')\big)^2}{16(L(N))^2}} \right].$$

\end{proposition}

\begin{proof}
Consider an intermediate point $(f'', H'')$ with Fenchel-Nielsen coordintaes $c'' = (l_i,\theta_i')$. We can construct a map $q':H \to H''$ with
\[\log K(q') \leq 3 C(N) d_{FN}(c,c'') \leq d_{FN}(c,c'),\]
 and a map $q'':H(c'') \to H(c')$ with 
 \[\log K(q'')\leq \frac{1}{L(N)} d_{FN}(c'',c') \sqrt{1 + \frac{\big(d_{FN}(c'',c')\big)^2}{16(L(N))^2}} \leq \frac{1}{L(N)} d_{FN}(c,c') \sqrt{1 + \frac{\big(d_{FN}(c,c')\big)^2}{16(L(N)) 2}}.\]

\end{proof}

Let $H_0$ be a complete hyperbolic structure on $S$, and suppose that $H_0$ is upper-bounded with respect to the given pair of pants decomposition $\mathcal{P}$. We  denote by $c_0$ the Fenchel-Nielsen coordinates of $H_0$ with respect to this pair of pants decomposition. From Proposition \ref{prop:iineq}, we deduce the following:

\begin{theorem}\label{th:Lipschitz1} We have an inclusion
$ \mathcal{T}_{FN}(H_0)\subset \mathcal{T}_{qc}(H_0)$. Moreover the inclusion map
$$i : \mathcal{T}_{FN}(H_0) \ni c \mapsto [f(c),H(c))] \in \mathcal{T}_{qc}(H_0)$$
is continuous and it is locally Lipschitz.
\end{theorem}

  Low let $[f,H]$ and $[f',H']$ be two elements in $\mathcal{T}_{qc}(H_0)$, with Fenchel-Nielsen coordinates $c = (l_i,\theta_i)$ and $c' = (l_i',\theta_i')$ respectively and let $N>0$ be such that such that $l_i, l_i' \leq N$ for all $i=,1,2,\ldots$.

Suppose that $d_{qc}([f,H],[f',H']) = D$. We will show that $d_{FN}(c,c')$ is bounded above by a constant that depends only on $D$.

\begin{lemma}   \label{lemma:controllength}
$$\sup_{i \in I}  \left|\log \left(\frac{l_H(C_i)}{l_{H'}(C_i)}\right)\right| \leq d_{qc}([f,H],[f',H'])$$
\end{lemma}
\begin{proof}
This is Wolpert's inequality, Theorem \ref{theorem:Wolpert} above.
\end{proof}
 
In order to show that $t=l_1\theta(T,U)$ is controlled by the quasiconformal dilatation $K(q)$, it is convenient to
lift the quasiconformal map $q$ to the universal cover,
and then consider its action on the ideal boundary.

We use the upper-plane model of hyperbolic plane.  Let $\Gamma$ and $\Gamma'$  be two Fuchsian groups  acting on $\mathbb{H}^2$ for $T$ and $U$ respectively.
We may assume that  $A(z)=\lambda z \ (\lambda=\exp(l_1) >1)$ belongs to $\Gamma$ and covers the geodesic $\alpha$.
We denote by $\widetilde{f^t}:\mathbb{H}^2\rightarrow \mathbb{H}^2$.
 the lift of the map $q$ to the universal cover.

  Note that $\widetilde{f^t}$ extends continuously to a homeomorphism on the ideal boundary $\partial \mathbb{H}^2$ and
the action of the extension on the limit set of the covering group is invariant under homotopy of $f$.
We also normalize $\widetilde{f^t}$ so that it fixes
$0,i,\infty$.

\begin{lemma}  \label{lemma:controltwist}
Given $[f,H],[f',H']$ as above, suppose that their length parameters are the same, i.e. $\forall i, l_i = l_i'$. Given a quasiconformal mapping $q:H \to H'$ such that $f' \circ q \circ f$ is homotopic to the identity, we have
$$d_{FN}(c,c') \leq \log K(q) $$
\end{lemma}
\begin{proof}    
 The result follows from Theorem  \ref{mixed}.
 
 \end{proof}

\begin{corollary}\label{cor:ident}
Given $[f,H],[f',H']$ as above and a quasiconformal mapping $q:H \to H'$ such that $f' \circ q \circ f$ is homotopic to the identity, we have
$$d_{FN}(c,c') \leq (2+3C(N)) \log K(q)$$  
\end{corollary}
\begin{proof}
Consider the set of Fenchel-Nielsen coordinates $c'' = (l_i,\theta_i')$. By Lemma \ref{lemma:controllength}, $d_{FN}(c',c'') < \log K(q)$. By Lemma \ref{lemma:difflengths}, there is a quasiconformal mapping $h:H' \to H(c'')$ with $K(h) < 3 C(N) \log K(q)$. Now $h \circ q: H \to H(c'')$ is a quasiconformal mapping with $\log K(h \circ q) < (1+3C(N)) K(q)$, between surfaces with the same length parameters. By Lemma \ref{lemma:controltwist}, $d_{FN}(c,c'') < (1+3C(N)) K(q)$. By the triangle inequality, $d_{FN}(c,c') < d_{FN}(c,c'')+d_{FN}(c'',c') < (2+3C(N)) K(q)$.    
\end{proof}

From Corollary \ref{cor:ident}, we deduce the following:

\begin{theorem}\label{th:Lipschitz2}
For every $[f,H] \in \mathcal{T}_{qc}(H_0)$ we have ${(l_i(f,H), \theta_i(f,H))}_{i \in I} \in \mathcal{T}_{FN}(H_0)$. Moreover, the identity map
$$j : \mathcal{T}_{qc}(H_0) \ni [f,H] \mapsto {(l_i(f,H), \theta_i(f,H))}_{i \in I} \in \mathcal{T}_{FN}(H_0)$$
is continuous and locally Lipschitz.
\end{theorem}

 Theorems \ref{th:Lipschitz1} and \ref{th:Lipschitz2} combined give the following

\begin{theorem}      \label{th:bilipschitz}
Let $H_0$ be a complete hyperbolic structure on $S$, and suppose that $H_0$ is upper-bounded. Denote by $c_0$ the Fenchel-Nielsen coordinates of $H_0$. Then the natural map
$$j : \mathcal{T}_{qc}(H_0) \ni [f,H] \mapsto {(l_i(f,H), \theta_i(f,H))}_{i \in I} \in \mathcal{T}_{FN}(H_0)$$
is a locally bi-Lipschitz homeomorphism. As $\mathcal{T}_{FN}(H_0)$ is isometric to the sequence space $l^\infty$, this gives a locally bi-Lipschitz homeomorphism between the Teichm\"uller space and $l^{\infty}$. 
\end{theorem}

A result by A. Fletcher (see \cite{Fletcher} and the survey in \cite{FMH}) says that the non-reduced quasiconformal Teichm\"uller space of any surface of infinite analytic type is locally bi-Lipschitz  to the sequence space $l^{\infty}$. Theorem  \ref{th:bilipschitz} above gives a global homeomorphism between the quasiconformal Teichm\"uller space $\mathcal{T}_{qc}(H_0)$ and the sequence space $l^{\infty}$. 

We note the following special case of Theorem \ref{th:bilipschitz}:

\begin{corollary}  \label{cor:equality-upper-bounded}
If the base hyperbolic metric $H_0$ is upper-bounded, then we have ${T}_{qc}(H_0)=\mathcal{T}_{FN}(H_0)$ (setwise).
\end{corollary}

It is interesting to notice, concerning the equality  $\mathcal{T}_{qc}(H_0)=\mathcal{T}_{FN}(H_0)$,  that the definition of the space $\mathcal{T}_{qc}(H_0)$ does not depend on the choice of the pair of pants decomposition, whereas the definition of 
$\mathcal{T}_{FN}(H_0)$ depends on such a choice.

\section{Examples and counter-examples}\label{s:examples}

In this section, we collect some examples of hyperbolic surfaces that show that some of the hypotheses in the results that we prove in this paper are necesssary, and that justify the choices that we made in some definitions.

The newt two examples concern the definition, for a hyperbolic structure $H$, of being upper-bounded. We first show that there are some hyperbolic surfaces that are not upper-bounded with reference to a certain pair of pants decomposition, but that are upper-bounded with reference to another one. We then show that there exist hyperbolic structures that are not upper-bounded with reference to any pair of pants decomposition.

\begin{example}\label{ex:pants1}
Consider the pairs of pants $P_n$ with one cusp and two boundary curves of length $1$ and $n$ respectively, and let $X_n$ be the surface constructed by gluing two copies of $P_n$ along the boundary component of length $n$, with zero twist parameter. You can construct a surface of infinite type $S$ by gluing one copy of $X_n$ for every natural number $n$. With respect to the pair of pants decomposition given by the pairs of pants $P_n$ (with two copies of $P_n$ for every $n$), the surface $S$ is not upper-bounded.  

Consider, in $P_n$, the shortest geodesic arc joining the boundary component of length $n$ to itself, and passing between the cusp and the other boundary component. By the second formula of Lemma \ref{hexagon}, the length $l$ of this arc satisfies 
$$ \cosh^2 l =  \displaystyle \coth^2 n + \frac{\cosh^2 1}{\sinh^2 n} + 2  \coth  n \frac{\cosh 1 }{\sinh n} \leq 3 \coth^2 1$$
By gluing the two copies of this arc embedded into $X_n$, we construct a simple closed curve splitting $X_n$ in two pairs of pants. Using these pairs of pants, we can construct a new pair of pants decomposition of $S$ making $S$ not upper-bounded.
\end{example}

\begin{example}\label{ex:pants2} We show now that there are some hyperbolic structures that are not upper-bounded with reference to any pair of pants decomposition. To see this note that by lemma \ref{lemma:complete} all upper-bounded surfaces are complete. In \cite{Basmajian} there are examples of noncomplete surfaces constructed by gluing pairs of pants. Consequently, these surfaces are not upper-bounded with reference to any pair of pants decomposition.
  \end{example}

In (\ref{def:FND}), we defined the Fenchel-Nielsen distance of two hyperbolic structures $x$ and $y$ by the formula

 \[{d_{FN}(x,y)=\sup_{i=1,2,\ldots} \max\left(\left\vert \log \frac{l_x(C_i)}{l_y(C_i)}\right\vert, \vert l_x(C_i)\theta_x(C_i)-l_y(C_i)\theta_y(C_i)\vert \right) }.
  \]
   
It would have also been possible to define a distance between $x$ and $y$ in which the  the  term 
$\displaystyle \left\vert \log \frac{l_x(C_i)}{l_y(C_i)}\right\vert$ in the above formula is replaced by $ \left\vert  l_{x_{n}}(C_i)- l_{y_{n}}(C_i)\right\vert$, and/or the term 
$\vert l_x(C_i)\theta_x(C_i)-l_y(C_i)\theta_y(C_i)\vert $ is replaced by $ \vert \theta_x(C_i)- \theta_y(C_i)\vert $. The first two examples below show that such metrics would have a different behaviour than the Fenchel-Nielsen metric as we defined it.

The next two examples concern the Fenchel-Nielsen distance. 

\begin{example}\label{ex:FN1}
Consider two sequences $(x_n)_{n=1,2,\ldots}$
 and $(y_n)_{n=1,2,\ldots}$ in  $\mathcal{T}_{qc}(H)$
satisfying the following

 \begin{equation*}
\left\{
     \begin{array}{lll}
          l_{x_{n}}(C_n)= 1/n, \ l_{x_{n}}(C_k)= 1 \hbox{ for } k\not=n, \\
{}\\
 \theta_{x_{n}}(C_k)= 0 \ \forall  k = 1,2,\ldots,  n = 1,2,\ldots
     \end{array}
  \right.
\end{equation*}

 \begin{equation*}
\left\{
     \begin{array}{lll}
          l_{y_{n}}(C_n)= 1/n, \ l_{y_{n}}(C_k)= 1 \hbox{ for } k\not=n, \\
{}\\
 \theta_{y_{n}}(C_n)= 2\pi, \   \ \theta_{y_{n}}(C_k)= 0 \hbox{ for } k\not=n
     \end{array}
  \right.
\end{equation*}

We have $d_{FN}(x_{n},y_{n})\to 0$ and, using for instance Proposition \ref{prop:iineq}, $d_{qc}(x_{n},y_{n})\to 0$), 
while 
\[ \sup_{i=1,2,\ldots} \max\left(\left\vert \log \frac{l_{x_{n}}(C_i)}{l_{y_{n}}(C_i)}\right\vert,  \vert \theta_{x_{n}}(C_i)- \theta_{y_{n}}(C_i)\vert \right)\]
is constant and equal to $2\pi$.

\end{example}

\begin{example}\label{ex:FN2}

Consider two sequences $(x_n^1)_{n=1,2,\ldots}$
 and $(x_n^2)_{n=1,2,\ldots}$ in  $\mathcal{T}_{qc}(H)$
satisfying the following

 \begin{equation*}
\left\{
     \begin{array}{lll}
          l_{x_{n}}(C_n)= 1/n, \ l_{x_{n}}(C_k)= 1 \hbox{ for } k\not=n, \\
{}\\
 \theta_{x_{n}}(C_k)= 0 \ \forall  k = 1,2,\ldots,  n = 1,2,\ldots

     \end{array}
  \right.
\end{equation*}

 \begin{equation*}
\left\{
     \begin{array}{lll}
          l_{y_{n}}(C_n)= 1/n^2, \ l_{y_{n}}(C_k)= 1 \hbox{ for } k\not=n, \\
{}\\
 \theta_{y_{n}}(C_k)= 0 \ \forall  k = 1,2,\ldots,  n = 1,2,\ldots
     \end{array}
  \right.
\end{equation*} We have
$d_{FN}(x_{n}, y_{n})\to\infty$,
while
\[ \sup_{i=1,2,\ldots} \max\left(\left\vert    l_{x_{n}}(C_i)- l_{y_{n}}(C_i)\right\vert, \vert l_{x_{n}}(C_i)\theta_{x_{n}}(C_i)-l_{y_{n}}(C_i)\theta_{y_{n}}(C_i)\vert \right)= \vert \frac{1}{n}-\frac{1}{n^2}\vert\to 0\]
as $n\to\infty$.
\end{example}

\end{document}